\documentclass[11pt, a4paper]{article}
\usepackage[top=2.1cm,bottom=2.1cm,left=2.5cm,right=2.5cm]{geometry}
\usepackage{amsmath,amsthm,mathrsfs,graphicx,amsfonts}
\usepackage{bm}
\usepackage{cases}
\usepackage{hyperref}
\usepackage{enumerate,enumitem}
\newcounter{propcounter}
\usepackage[english]{babel}
\usepackage{amsmath,amsthm}
\usepackage{amssymb}
\usepackage{amsthm}
\usepackage{amsfonts}
\usepackage{latexsym}
\usepackage{graphicx}
\usepackage{txfonts}
\usepackage[numbers,sort&compress]{natbib}
\usepackage[natural]{xcolor}
\usepackage{rotating}
\usepackage{mathtools}
\usepackage{amsmath}
\usepackage{appendix}
\numberwithin{figure}{section}
\usepackage{todonotes}

\newtheorem{theorem}{Theorem}[section]
\newtheorem{lemma}[theorem]{Lemma}
\newtheorem{corollary}[theorem]{Corollary}
\newtheorem{proposition}[theorem]{Proposition}
\newtheorem{definition}[theorem]{Definition}

\newtheorem{claim}[theorem]{Claim}

\def\mod{{\rm mod}}

\title{Balanced subdivisions and cycles lengths in $K_{s, t}$-free graphs
}{}

\author{
        Jianfeng Hou$^{1}$\thanks{Research was supported by National Key R\&D Program of China (Grant No. 2023YFA1010202), National Natural Science Foundation of China (Grant No. 12071077), the Central Guidance on Local Science and Technology Development Fund of Fujian Province (Grant No. 2023L3003). Email: \texttt{jfhou@fzu.edu.cn}},~ Yindong Jin$^{2}$\thanks{Corresponding author. E-mail: \texttt{ydjin7@163.com}},~ Donglei Yang$^{3}$\thanks{ Supported
        by Natural Science Foundation of China (12101365) and by Natural Science Foundation of Shandong Province (ZR2021QA029). Email: \texttt{dlyang@sdu.edu.cn} },~Fan Yang$^4$\thanks{  Supported by Natural Science Foundation of China (12301447), by Natural Science Foundation of Shandong Province (ZR2024QA056), by China Scholarship Council and IBS-R029-C4. Email: \texttt{fyang@sdu.edu.cn}} ~ \\
{\small 1 Center for Discrete Mathematics, Fuzhou University, Fujian, China}\\
{\small 2  School of Mathematics and Statistics, Linyi University, Shandong, China}\\
{\small 3 School of Mathematics, Shandong University, Shandong, China}\\
    {\small 4 Data Science Institute, Shandong University, Shandong, China \& }
\\
{\small Extremal Combinatorics and Probability Group (ECOPRO), Institute for Basic Science (IBS), Daejeon, South Korea.}
      }  

\date{}

\begin{document}
\linespread{1.2}
\openup 1.0\jot
\date{}\maketitle
\begin{abstract}

Confirming a conjecture of Mader from 1999, Liu and Montgomery [J. Lond. Math. Soc., 2017] showed that for integers $t\ge s\ge2$, every  $K_{s, t}$-free graph with average degree $d$ contains a  subdivision of a clique with at least $\Omega(d^{\frac{s}{2(s-1)}})$ vertices. We give an improvement  by showing that such a graph  contains a balanced clique subdivision of the same order,  where a balanced subdivision means that each edge is subdivided the same number of times.

In 1975, Erd\H{o}s asked  whether the sum of the reciprocals of the cycle lengths in a graph with infinite average degree $d$ is necessarily infinite. Recently, Liu and Montgomery [J. Amer. Math. Soc., 2023] confirmed the asymptotically sharp lower bound $\left(\frac{1}{2} -o_d(1)\right) \log d$ on the reciprocals of the cycle lengths. In this paper,  we extend the lower bound to $\left(\frac{s}{2(s-1)} -o_d(1)\right) \log d$ for $K_{s, t}$-free graphs.

Both proofs of our results use the graph   sublinear expansion property as well as some novel structural techniques.

\bigskip

\end{abstract}
\baselineskip 20pt

\section{Introduction}

Given a graph $H$, a \emph{subdivision} of $H$, denoted by $\mathrm{T}H$,  is a graph obtained from $H$ by replacing some edges of $G$ with internally vertex-disjoint paths. The original vertices of $H$ are the \emph{branch
vertices} of $\mathrm{T}H$, and the new vertices are called \emph{subdividing vertices}. A typical example is  the seminal result of Kuratowski in 1930 stating that a graph is
planar if and only if it does not contain a  $\mathrm{T}K_5$ or $\mathrm{T}K_{3,3}$, where $K_s$ is a complete graph of order $s$, and $K_{s,t}$ is a complete  bipartite graph with part sizes $s$ and $t$.

A fundamental extremal  problem in this topic is to find the smallest average degree $d:=d(k)$ of graphs forcing the appearance of $\mathrm{T}K_k$ as a subgraph.   This was initially studied by  Mader  \cite{1967mader} who showed that such a $d(k)$ exists. Later, Mader \cite{1967mader}, and independently Erd\H{o}s and Hajnal \cite{erd} conjectured that $d(k) = O(k^2)$. After some further results by Mader \cite{mader}, the conjecture was confirmed by Bollob\'{a}s and Thomason \cite{bollo}, and independently by Koml\'{o}s and Szemer\'{e}di  \cite{1994jk, 1996jk}. For graphs that are far from disjoint union of complete bipartite graphs, one can actually do better. Let $C_k$ denote a cycle of length $k$ for an integer $k\ge 3$. Given a graph $H$, we say that a graph is \emph{$H$-free} if it does not contain $H$ as a subgraph. Mader \cite{1999mader} conjectured that for every $C_4$-free graph with average degree $d$ contains a $\mathrm{T}K_{\Omega(d)}$. K\"{u}hn and Osthus \cite{kuhn2002, kuhn2006} proved that every graph with sufficiently large girth contains a subdivision of a clique with order linear in its minimum degree. They \cite{2004kd} also found a $\mathrm{T}K_{d/ \log^{12} d}$ in $C_4$-free graphs with average degree $d$. Balogh, Liu and Sharifzadeh \cite{2015c6} proved that each $C_{2k}$-free graph with average degree $d$ contains a $\mathrm{T}K_{\Omega(d)}$ for $k\ge 3$. Using new constructions of clique subdivisions, Mader's conjecture was finally settled by Liu and Montgomery \cite{2017liu}. In fact, they considered clique subdivisions in $K_{s,t}$-free graphs and  proved the following more general result. 

\begin{theorem}[\cite{2017liu}]\label{Kst-free-subdivision}
For all integers $t\ge s\ge 2$, there exists $c = c(s, t)>0$ so that the following holds for every $d > 0$. Every $K_{s, t}$-free graph $G$ with average degree $d$ contains a  $\mathrm{T}K_{cd^{s/2(s-1)}}$.
\end{theorem}

\subsection{Balanced subdivision}

A natural extension, proposed by Thomassen \cite{1984thom, 1985thom1, 1985thom2}, is to find balanced clique subdivisions in graphs under average degree condition. For a graph $H$, an \emph{$\ell$-balanced subdivision} of $H$, denoted by $\mathrm{T} H^{(\ell)}$, is a graph obtained from $H$ by replacing each of its edges with internally vertex-disjoint paths of length exactly $\ell$. A graph  has a \emph{balanced  $\mathrm{T} H$} if it has a $\mathrm{T} H^{(\ell)}$ for some $\ell\ge 1$. Motivated by Mader's theorem, Thomassen \cite{1985thom1} conjectured that there exists a $d:=d(k)$ such that every graph with average degree at least $d$ contains a balanced $\mathrm{T} K_k$. The conjecture was confirmed by Liu and Montgomery \cite{2023liu}. Naturally, the next step is to better estimate $d(k)$. Let $G$ be an  $n$-vertex graph with average degree $d$.  Wang \cite{2023wang} proved that $G$ contains a balanced  $\mathrm{T} K_k$, where $k =d^{1/2-o(1)}$. Simultaneously, Luan, Tang, Wang and Yang \cite{2023yang}, independently Fern\'{a}ndez, Hyde, Liu, Pikhurko and Wu \cite{2023wu} proved that $G$ contains a balanced $\mathrm{T}K_{\Omega(\sqrt{d})}$.  For $C_4$-free graphs, the authors in \cite{2023yang} further obtained the following linear bound. 

\begin{theorem}[\cite{2023yang}]\label{C4-free-balanced}
Every $C_4$-free graph with average degree $d$ contains a balanced  $\mathrm{T} K_{\Omega( d)}$.
\end{theorem}

Our first result concerns the existence of balanced clique subdivisions in $K_{s,t}$-free graphs, generalizing Theorem \ref{Kst-free-subdivision} to the balanced setting. Moreover, compared to Theorem \ref{C4-free-balanced}, we deal with the $K_{s,t}$-free condition for general $t\geq s\geq 2$, which requires additional techniques and careful analysis. 

\begin{theorem}\label{main-thm}
For all integers $t\ge s\ge 2$, there exists a constant $c = c(s, t)$ such that the following holds for every $d > 0$. Every $K_{s, t}$-free graph $G$ with average degree $d$ contains a balanced $\mathrm{T}K_{cd^{\frac{s}{2(s-1)}}}$.
\end{theorem}

\subsection{Cycles  with consecutive even lengths}

An interesting direction concerning cycles in graphs is to study the distribution of cycle lengths. For a graph $G$, let $\mathcal{C}(G)$ be the set of cycle lengths. A classical  conjecture in this field, posed by Erd\H{o}s and Hajnal \cite{erd1966},  states that 
\begin{equation}\label{Erdos-hajnal-conjeture}
\sum_{\ell\in \mathcal{C}(G)}\frac{1}{\ell}\rightarrow\infty \text{~ ~  as~~~  } \chi(G)\rightarrow\infty, 
\end{equation}
where $\chi(G)$ denotes the chromatic number of a graph $G$. As noted by Erd\H{o}s, they felt  that \eqref{Erdos-hajnal-conjeture} holds under the weaker condition that the average degree of $G$ tends to infinity.  Let $G$ be a graph with average degree $d$. Confirming this stronger conjecture,  Gy\'{a}rf\'{a}s, Koml\'{o}s and Szemer\'{e}di \cite{g1984} proved  that $\sum_{\ell\in C(G)}1/\ell=\Omega_d(\log d)$. Considering cycles in  complete balanced
bipartite graphs,  Erd\H{o}s \cite{erd1975} have previously stated that the correct asymptotic lower bound was likely  $(1/2 +o_d(1)) \log d$. Recently, the conjecture was confirmed by Liu and Montgomery \cite{2023liu}.

To prove Erd\H{o}s conjecture, Liu and Montgomery \cite{2023liu} gave a stronger result by studying cycles with consecutive even lengths under the average degree condition. This is another interesting  direction in the study of the distribution of cycle lengths. Bondy and Vince \cite{Bondy1998} proved that every graph with minimum degree at least $3$ contains two cycles whose lengths differ by one or two. It was extended to graphs with large minimum degree by Fan \cite{fan2002}, who showed  that every graph with minimum degree at least $3k-2$ contains $k$ cycles with consecutive even lengths or consecutive odd lengths. Under the average degree condition,  Verstra\"{e}te \cite{ver2000}  showed that every  graph with  average degree at least $8k$ and even girth $g$ has at least $(g/2-1)k$ cycles with
consecutive even lengths.  In 2008, Sudakov and Verstra\"{e}te \cite{sud} proved  that if a graph $G$ has average degree $192(k+1)$ and girth $g$, then $G$ has $k^{\lfloor\frac{g-1}{2}\rfloor}$ cycles with consecutive even lengths. Finding cycles with consecutive odd lengths seems difficult. A breakthrough was given by Ma \cite{ma2016}, who showed that there exists an absolute constant $c>0$ such that for every natural number $k$, every non-bipartite 2-connected graph $G$ with average degree at least $ck$ and girth $g$ contains at least $k^{\lfloor\frac{g-1}{2}\rfloor}$ cycles with consecutive odd lengths. For  recent results in this field, we refer the interested readers to \cite{gjmj2020, jtmj2023, lf2021, gjhq2022, gmj2021, yy2025}

To the sum of the reciprocals of the distinct cycle lengths, Liu and Montgomery \cite{2023liu} have not only found a long interval of consecutive even numbers in $\mathcal{C}(G)$, but also  relatively controlled the upper and lower bound of the interval. 

\begin{theorem}[\cite{2023liu}]
There is $d_0 > 0$ such that the following holds. If $G$ is a graph with average degree $d \ge d_0$, then there is some $\ell\ge \frac{d}{10 \log^{12} d}$ such that $\mathcal{C}(G)$ contains every even integer in $[\log^8\ell, \ell]$.
\end{theorem}

Our second result is to construct even cycles whilst controlling their length in $K_{s,t}$-free graphs, which, thanks to a seminar result of Koml\'{o}s--Szemer\'{e}di, is reformulated  and discussed in expanders \footnote{The definition of expander is deferred to Subsection~\ref{subsec:expander}} as follows.

\begin{theorem}\label{main-thm-2}
Suppose $t\ge s\ge 2$ are integers and $0< \varepsilon_2<1/5$, there exist $d_0, \varepsilon_1, \varepsilon>0$ such that the following holds. Let $G$ be an $n$-vertex $K_{s, t}$-free bipartite $(\varepsilon_1, \varepsilon_2d^{\frac{s}{s-1}})$-expander with average degree $d\ge d_0$. Then $\mathcal{C}(G)$ contains every even integer in 
\begin{enumerate}
\item[$(1)$] $\left[\frac{4}{\varepsilon_1}\log^3\left(\frac{15n}{\varepsilon_2 d^{\frac{s}{s-1}}}\right),  \frac{d^{\frac{s}{s-1}}}{(288t^{1/s})^{s/(s-1)}}\right]$  if $d>\varepsilon n^{\frac{s-1}{s}}$,
\item[$(2)$] $\left[300\cdot\log^8\frac{n}{d^{s/(s-1)}}, \frac{d^{s/(s-1)}}{100}\cdot\log^{12}\frac{n}{d^{s/(s-1)}}\right]$ if $\log^{200} n\le d\le\varepsilon n^{\frac{s-1}{s}}$,
\item[$(3)$] $\left[\log^7n, \frac{n}{\log^{12} n} \right]$ if $d< \log^{200} n$.
\end{enumerate}
\end{theorem}

By combining Theorem \ref{main-thm-2} and a classical result of Koml\'{o}s and Szemer\'{e}di \cite{1996jk} asserting that every graph $G$ contains a sublinear expander almost as dense as $G$, we obtain the following extension of the result of Liu and Montgomery \cite{2017liu}.
\begin{corollary}\label{coro1}
Let $t\ge s\ge 2$ be integers. If $G$ is a $K_{s, t}$-free graph with average degree $d$, then
\[
\sum_{\ell\in\mathcal{C}(G)}\frac{1}{\ell}\ge\left(\frac{s}{2(s-1)}-o_d(1)\right)\log d.
\]
\end{corollary}

\subsection{Ideas and organization}
Both proofs of our results pass to a subgraph that is a sublinear expander to enjoy the expansion property. Although the linear expansion property has been studied extensively, Koml\'os and Szemer\'edi \cite{1994jk, 1996jk} introduced a sublinear expansion property, which forms the basis of our proofs.

To find  balanced clique subdivisions in general graphs, a key idea is to build a large subgraph called \emph{adjuster} by joining many small adjusters \footnote{The definition of adjuster is deferred to Section~\ref{pf-main-lemmas}.}. It can help us to find a large balanced clique subdivision in relatively  sparse expanders. It is also  the main idea used in \cite{2023yang, 2023wu}.  For dense expanders (e.g. $d=\Omega(\sqrt{n})$ for the $C_4$-free case),  the desired balanced clique subdivision can be found directly using dependent random choice (see \cite{2011fox} for details) in \cite{2023yang, 2023wu}. However, it is still not clear how to adapt their approach for the $K_{s,t}$-free graphs for general integers $s,t\ge 2$.  To overcome this issue, we use an idea from the work of Liu and Montgomery \cite{2017liu}. The key is to build tree-like structures that act as large degree vertices and to connect them using vertex-disjoint paths of the same length.

For cycles with consecutive even lengths in $K_{s,t}$-free graphs, different from \cite{2023liu}, we divide our proof into three cases according to the average degree $d$ of $G$. If $d\le\log^{200}n$ or $d\ge\varepsilon n^{\frac{s-1}{s}}$, then we adopt the methods in \cite{2023liu} and \cite{sud}, respectively. Otherwise, we use a special structure called \textit{unit} from \cite{2017liu}, and follow the strategy of Liu and Montgomery \cite{2023liu} by finding a short path (specifically, in \cite{2023liu}, they just choose an edge) with two endpoints centered with internally vertex-disjoint units. 

The remainder of the paper is organized as follows. In Section \ref{sec2}, we first introduce some
necessary notions and results. Then, we give a proof of Theorem \ref{main-thm} through two key lemmas (Lemmas \ref{liu} and  \ref{yang}), whose proofs can be found in Section \ref{pf-main-lemmas}. In Section \ref{sec6}, we prove Theorem \ref{main-thm-2}. The final section  offers some concluding remarks.



\section{Preliminaries}\label{sec2}
Let $G$ be a graph. We use $|G|$ and $e(G)$ to denote the number of vertices and edges of $G$, respectively. Let $d(G)=2e(G)/|G|$ be the \emph{average degree} of $G$. Given two distinct vertex sets $X$ and $Y$, let $E_G(X, Y)$ be the set of edges between $X$ and $Y$, and $e_G(X, Y):= |E(X, Y)|$. Denote by $G[X]$ the induced subgraph of $G$ on $X$, and write $G-X$ for the induced subgraph $G[V(G)\backslash X]$.  We define the \emph{external neighborhood} of
$X$ in $G$ to be 
$$N_G(X):=\left\{y\in V(G)\setminus X: \text{ there exists } x\in X \text{ such that } xy\in E(G)\right\}.$$
Let $P$ be a path with endvertices $x$ and $y$. We say $P$ is an \emph{$(x,y)$-path}. The \emph{length} of $P$, denoted by $\ell(P)$,  is  the number of edges in it. If $P$ is an $(x,y)$-path, and $P\cap(X\cup Y)=\{x,y\}$, where $x\in X, y\in Y$, then we say $P$ is a \emph{path from  $X$ to  $Y$}. We will drop the subscript when the confusion is unlikely. Usually, we write $[k]:=\{1,\ldots,k\}$. For convenience, as it is standard in the literature, we will usually pretend that large numbers are integers to avoid using essentially irrelevant floor and ceiling symbols. Throughout the paper, the symbol $e$ denotes the standard natural constant in mathematics and we write $\log$ for the natural logarithm.

\subsection{Koml\'{o}s--Szemer\'{e}di  expander}\label{subsec:expander}

We use the following notion of expander introduced by Koml\'{o}s and Szemer\'{e}di in \cite{1994jk, 1996jk}.

\begin{definition}\rm{(Sublinear expander)}.\label{subk} 
Let $\varepsilon_1 > 0$ and $k > 0$. A graph $G$ is an $(\varepsilon_1, k)$-\emph{expander} if $|N(X)| \ge \varepsilon(|X|, \varepsilon_1, k) \cdot |X|$ for all $X \subseteq V (G)$ with $k/2 \le |X| \le |G|/2$, where
$$
\varepsilon(x, \varepsilon_1, k): =
\left\{
    \begin{array}{lc}
        0&{\rm if}\ x<k/5,\\
        \frac{\varepsilon_1}{\log^2 (15x/k)}&{\rm if}\ x\ge k/5.\\
    \end{array}
\right.
$$
\end{definition}
Whenever the choices of $\varepsilon_1$, $k$ are clear, we omit them and write $\varepsilon(x)$ for $\varepsilon(x, \varepsilon_1, k)$. The key property of sublinear expanders is that one  can connect two large sets of vertices using a short path even after removing a smaller vertex set.

\begin{lemma}[\cite{1996jk}]\label{diamter}
Let $\varepsilon_1, k > 0$. If $G$ is an $n$-vertex $(\varepsilon_1, k)$-expander, then any two vertex sets, each of size at least $x \ge k$, are of distance at most $\frac{2}{\varepsilon_1}\log^3\left(\frac{15n}{k}\right)$ apart. This remains true even after deleting $x \cdot \varepsilon(x)/4$ arbitrary vertices from $G$.
\end{lemma}

In 1996, Koml\'{o}s and Szemer\'{e}di \cite{1996jk} showed that every graph $G$ with average degree $d$ contains an expander with average degree and minimum degree linear in  $d$.
\begin{theorem}[\cite{1996jk}]\label{subgraph-expander}
There exists $\varepsilon_1 > 0$ such that the following holds for every $k > 0$. Every graph $G$ has an $(\varepsilon_1, k)$-expander $H$ with $d(H) \ge d(G)/2$ and $\delta(H) \ge d(H)/2$.
\end{theorem}

We remark that $H$ might be much smaller than $G$ in Theorem \ref{subgraph-expander}. Note that any graph $G$ has a  bipartite subgraph $H$ with $d(H) \ge d(G)/2$ by considering the  bipartite subgraph with a maximum number of edges of $G$.  This together with Theorem \ref{subgraph-expander} yields the following corollary.

\begin{corollary}\label{coro}
There exists $\varepsilon_1 > 0$ such that the following holds for every $k > 0$ and $d \in \mathbb{N}$. Every graph $G$ with $d(G) \ge 8d$ has a bipartite $(\varepsilon_1, k)$-expander $H$ with $\delta(H) \ge d$.
\end{corollary}

\subsection{Bipartite $K_{s,t}$-free graphs}
In this subsection, we collect some well-known results about $K_{s,t}$-free graphs.  
A basic problem in extremal graph theory is to determine the maximum possible number of edges in an $n$-vertex $K_{s,t}$-free graph. For a given graph $H$, we define $\text{ex}(n,H)$ to be the maximum number of edges in an $n$-vertex $H$-free graph. K\H{o}v\'{a}ri, S\'{o}s and Tur\'{a}n \cite{1954kova} gave the following theorem.
\begin{theorem}[\cite{1954kova}]\label{turan-kst}
For every integers $1\le s\le t$, $ex(n,K_{s,t})\le t^{1/s}n^{2-1/s}$.
\end{theorem}

The following analogous result of K\H{o}v\'{a}ri, S\'{o}s and Tur\'{a}n \cite{1954kova} provides an upper bound on the average degree in one part of a $K_{s,t}$-free bipartite graph, which will also be used in our arguments.

\begin{lemma} [\cite{1954kova}]\label{1954kova}
Let $G = (A, B)$ be a bipartite graph that does not contain a copy of $K_{s,t}$ with $t$ vertices in $A$ and s vertices in $B$. Then
\[
|A|\tbinom{\bar{d}(A)}{s}\le t\tbinom{|B|}{s},
\]
where $\bar{d}(A) =\sum_{v\in A} d(v)/|A|$ is the average degree of the vertices in $A$.
\end{lemma}

As an immediate consequence of Lemma~\ref{1954kova}, Liu and Montgomery \cite{2017liu} obtained, given the  minimum degree condition in one part, the lower bound on the size of the other part in a $K_{s,t}$-free bipartite graph.

\begin{corollary}[\cite{2017liu}]\label{coro2.6}
Let $G = (A, B)$ be a bipartite graph that does not contain a copy of $K_{s, t}$ with $t$ vertices in $A$ and $s$ vertices in $B$, and in which every vertex in $A$ has at least $\delta$ neighbors in $B$. Then $|B| \ge \frac{\delta}{et}|A|^{1/s}.$
\end{corollary}

Theorem~\ref{1996furedi}, due to F\"{u}redi, provides an explicit upper bound on the Zarankiewicz number $Z(m,n,s,t)$, which is the maximum number of edges in a $K_{s,t}$-free bipartite graph with parts of sizes $m$ and $n$. In our context, it allows us to maintain the edge density of the remaining subgraph after removing a reasonable-sized vertex set.

\begin{theorem}[\cite{1996furedi}]\label{1996furedi}
For all $m\ge s$, $n\ge t$ and  $ t\ge s \ge 2$, we have
\[
Z(m, n, s, t)\le (t-s+1)^{1/s}mn^{1-1/s}+(s-1)n^{2-2/s}+(s-2)m.
\]
\end{theorem}

\subsection{Main tools and overview}\label{sec3}
In this subsection, we provide an overview of the proof of Theorem \ref{main-thm}. For simplicity, in the remainder of the paper we write
\begin{equation*}
\eta:=d^{\frac{s}{2(s-1)}}. 
\end{equation*}
Let $G$ be a $K_{s, t}$-free graph with average degree $d$. Then by Corollary \ref{coro} we can find a bipartite expander $G'$ in $G$. If $G'$ contains a $\mathrm{T}K_{c\eta}^\ell$  for some $c>0$, then we are done. Otherwise, Lemma~\ref{liu} would imply that $G'$ contains a subexpander $H$ with a large order and high minimum degree. Based on this, we further divide the remaining  proof into two cases according to whether $H$ is dense or not. We will use Lemma \ref{yang} to handle the dense case, and the sparse case is covered in Lemma \ref{le3}.

While the Koml\'os--Szemer\'edi lemma only asserts the existence of an expander without specifying size, Lemma~\ref{liu}  provides either a large-sized expander in a bipartite $K_{s,t}$-free graph or a large balanced clique subdivision as desired.

\begin{lemma}\label{liu}
Suppose $t\ge s\ge2$ are integers and $1/d\ll c\ll 1/K\ll\varepsilon_2\ll\varepsilon_1, 1/s, 1/t$. Let $G$ be an $n$-vertex bipartite  $K_{s,t}$-free   $(\varepsilon_1, \varepsilon_2\eta^2)$-expander with $\delta(G) \ge d/8$. Then one of the following holds: 
\begin{enumerate}
\item[$(1)$] There exists a subgraph $H$ of $G$ with $\delta(H) \ge \frac{d}{16}$ and $ |H| \ge K\eta^2$, which is an $(\varepsilon_1/2, \varepsilon_2\eta^2)$-expander. 
\item[$(2)$] $G$ contains a $\mathrm{T}K^{(\ell)}_{c\eta}$ for some $\ell\in \mathbb{N}$.
\end{enumerate}
\end{lemma}

\begin{lemma}\label{yang}
Suppose $t\ge s\ge2$ and $1/d\ll c\ll1/K\ll\varepsilon_1, \varepsilon_2, 1/s,  1/t$. Suppose $n\ge K\eta^2$ and   $d\ge\log^{200}n$. Let $G$ be an $n$-vertex bipartite $K_{s,t}$-free   $(\varepsilon_1, \varepsilon_2\eta^2)$-expander with $\delta(G) \ge \frac{d}{16}$. Then $G$ contains a $\mathrm{T}K^{(\ell)}_ {c\eta}$ for some $\ell\in \mathbb{N}$.
\end{lemma}
For sparse expanders, we use a result of Wang \cite{2023wang} to find a balanced clique subdivision of linear size.
\begin{lemma}\label{le3}{\rm(Lemma 1.3 in \cite{2023wang})}.
Suppose $1/d\ll c\ll\varepsilon_1, \varepsilon_2$ and  $d<\log^{200}n$. Let $G$ be an $n$-vertex bipartite $\mathrm{T}K^{(2)}_{d/2}$-free  $(\varepsilon_1, \varepsilon_2d)$-expander with $\delta(G) \ge d$. Then $G$ contains a $\mathrm{T}K^{(\ell)}_ {cd}$ for some $\ell\in \mathbb{N}$.
\end{lemma}
We also use the following result in \cite{2017liu}. 
\begin{proposition}{\rm(Proposition 5.2  in \cite{2017liu})}.\label{prop}
Let $t\ge s\ge2$ be integers and choose $0<\varepsilon_1<1$ and $0<\varepsilon_2<\frac{1}{10^5t}$. If $G$ is a $K_{s, t}$-free, $(\varepsilon_1, \varepsilon_2\eta^2)$-expander with $\delta(G)\ge \frac{d}{16}$, then $G$ is also an $(\varepsilon_1, \varepsilon_2d)$-expander.
\end{proposition}
Now we are ready to prove Theorem \ref{main-thm}.

\begin{proof}[Proof of Theorem \ref{main-thm}] We choose $1/d\ll c\ll1/K\ll\varepsilon_2\ll\varepsilon_1, 1/s, 1/t$. 
Let $G$  be an $n$-vertex $K_{s,t}$-free graph with average degree $d > 0$. By Corollary \ref{coro}, $G$ contains a bipartite subgraph $G_1$ with $\delta(G_1) \ge d(G_1)/2 \ge d/8$ which is an $(\varepsilon_1, \varepsilon_2\eta^2)$-expander.

By Lemma \ref{liu}, we can find either a balanced $\mathrm{T}K_{c_{{\ref{liu}}}\eta}$ in $G_1$, or a subgraph $H$ in $G_1$ with $\delta(H) \ge \frac{d}{16}$ and $|H| \ge K\eta^2$, which is an $(\varepsilon_1/2, \varepsilon_2\eta^2)$-expander. If $d \ge \log^{200} |H|$, then by the choice of $\varepsilon_2$ and $K$ with Lemma \ref{yang}, $H$ contains a  balanced $\mathrm{T}K_{c_{{\ref{yang}}}\eta}$. Otherwise, $d < \log^{200} |H|$. Since $H$ (as a subgraph of $G$) is $K_{s,t}$-free, Proposition \ref{prop} implies that $H$ is also an $(\varepsilon_1/2, \varepsilon_2d)$-expander and we may assume $H$ is $\mathrm{T}K_{d/2}^{(2)}$-free. Then Lemma \ref{le3} yields that $H$ contains a  balanced $\mathrm{T}K_{c_{{\ref{le3}}}\eta}$. Thus, we can take $c=\text{min}\{c_{{\ref{liu}}}, c_{{\ref{yang}}}, c_{{\ref{le3}}}\}$ and then $G$ contains a balanced $\mathrm{T}K _{c\eta}$.
\end{proof}

\section{Proofs of Lemmas \ref{liu} and \ref{yang}}\label{pf-main-lemmas}

\subsection{Unit and adjuster}\label{sub:definitions}

\begin{figure}[h]
\begin{center}
\scalebox {1.50}[1.50]{\includegraphics {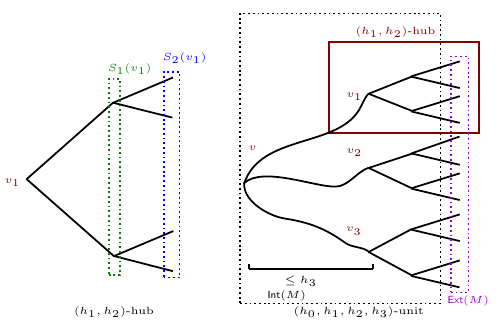}}
\end{center}
\caption{An $(h_0,h_1,h_2,h_3)$-unit with $h_0=3$ and  $h_1=h_2=2$, an $(h_1,h_2)$-hub with $h_1=h_2=2$.}
\end{figure}
In this subsection,  we introduce some notions from Liu and   Montgomery \cite{2017liu, 2023liu}.
\begin{definition}
\rm
(Hub \cite{2017liu}) Given integers $h_1, h_2 > 0$, an \emph{$(h_1, h_2)$-hub} is a graph consisting of a vertex $v_1$ called \emph{center}, a set $S_1(v_1) \subseteq N(v_1)$ of size $h_1$, and pairwise disjoint sets $S_1(z) \subseteq N(z)\backslash\{v_1\}$ of size $h_2$ for each $z \in S_1(v_1)$. Denote by $H(v_1)$ a hub with center $v_1$ and write $B_1(v_1) = \{v_1\}\cup S_1(v_1)$
and $S_2(v_1) = \cup _{z\in S_1(v_1)} S_1(z)$. For any $z \in S_1(v_1)$, write $B_1(z) = \{z\} \cup S_1(z)$.
\end{definition}
\begin{definition}
\rm
(Unit \cite{2017liu}) Given integers $h_0$, $h_1$, $h_2$, $h_3 > 0$, an \emph{$(h_0, h_1, h_2, h_3)$-unit} $M$ is a graph
consisting of a \emph{core vertex} $v$, $h_0$ vertex-disjoint $(h_1, h_2)$-hubs $H(v_1), \ldots, H(v_{h_0})$ and pairwise disjoint $(v, v_j)$-paths of length at most $h_3$. By the \emph{exterior} of the unit, denoted $\mathsf{Ext}(M)$, we mean $\bigcup_{j=1}^{h_0} S_2(v_j)$. Denote by $\mathsf{Int}(M):=V (M)\backslash \mathsf {Ext}(M)$ the \emph{interior} of the unit.
\end{definition}
\begin{definition}
\rm
(Expansion \cite{2017liu}) Given a vertex $v$ in a graph $F$, $F$ is a \emph{$(D, m)$-expansion} of $v$
if $|F| = D$ and $v$ is at a distance of at most $m$ in $F$ from any other vertex of $F$.
\end{definition}
\begin{proposition} \rm(\cite{2023liu})\label{3.4}
 Let $D$, $m \in \mathbb{N}$ and $1 \le D_0 \le D$. Then, any 
$(D, m)$-expansion of $v$ contains a subgraph which is a $(D_0, m)$-expansion of $v$.
\end{proposition}
Liu and Montgomery \cite{2023liu} introduced a gadget called \emph{adjuster}, which plays a key role in finding subdividing paths of equal length in our proof.
\begin{definition}
\rm
(Adjuster \cite{2023liu})\label{defad} 
For $D, m, k \in \mathbb{N}$, a $(D, m, k)$-adjuster $\mathcal{A}= (v_1, F_1, v_2, F_2, A)$ in a graph $G$ consists of core vertices $v_1$, $v_2 \in V (G)$, graphs $F_1$, $F_2 \subseteq G$ and a center vertex set $A \subseteq V (G)$ such that the following hold for some $\ell \in \mathbb{N}$.
\stepcounter{propcounter}
\begin{enumerate}[label = ({\bfseries \Alph{propcounter}\arabic{enumi}})]  
\rm\item\label{Ad1} The subsets $A$, $V (F_1)$ and $V (F_2)$ are pairwise disjoint.

\rm\item\label{Ad2} For each $i \in [2]$, $F_i$ is a $(D, m)$-expansion of $v_i$.

\rm\item\label{Ad3} $|A| \le 10mk$. 

\rm\item\label{Ad4} For each $i \in \{0, 1, \ldots, k\}$, there is a $(v_1, v_2)$-path in $G[A \cup \{v_1, v_2\}]$ of length $\ell + 2i$.
\end{enumerate}
\end{definition}
We call the smallest such $\ell$  in Definition \ref{defad} the \emph{initial length} of the adjuster and denote it by
$\ell(\mathcal{A})$. Note that $\ell(\mathcal{A}) \le |A|+1 \le 10mk+1$. For convenience, we often call a $(D, m, 1)$-adjuster
a \emph{simple adjuster}. Let $V (\mathcal{A}) = V (F_1) \cup V (F_2) \cup A$.

\subsection{Proof of Lemma \ref{liu}}\label{sec4}

Here, we give an outline of the proof of Lemma \ref{liu}. Assume that $G$ is as in Lemma \ref{liu}. Let $Z$ be the set of vertices in $G$ with high degree. If  $|Z|$ is large, then we find a desired $\mathrm{T}K^{(\ell)}_{c\eta}$  using some vertices in $Z$ as branch vertices, and greedily connecting each pair of those vertices with vertex-disjoint paths of  length $\ell$. Otherwise, let $H=G-Z$.  We show that $|H|\geq K\eta^2$ and $H$ almost has the similar expansion property. 

We first state the following two lemmas and then prove Lemma \ref{liu}. In subsection \ref{sub:Building-adjuster}, we prove Lemma \ref{adjuster1}. Lemma \ref{path1} uses a slightly different construction from Theorem 2.7 in Liu and Montgomery \cite{2023liu}, and we leave its proof in Appendix \ref{3.7-le}. Recall that $\eta:=d^{\frac{s}{2(s-1)}}$. Throughout the rest of the paper, we take 
\[
m:=\log\frac{15n}{\varepsilon_2\eta^2}
\]
and let $\ell$ be the smallest even integer larger than $m^{10}$ in the following of this subsection.

\begin{lemma}\label{adjuster1}
Suppose $t\ge s\ge2$ are integers,  $1/d\ll c\ll 1/K\ll\varepsilon_2\ll\varepsilon_1, 1/s, 1/t$.  Let $D=\max\left\{\frac{c^2m^{19}\eta^2}{10^{20}},~\frac{\varepsilon_2\eta^2}{10^{20}}\right\}$ and $G$ be an $n$-vertex $K_{s,t}$-free  $(\varepsilon_1, \varepsilon_2\eta^2)$-expander with $\delta(G) \ge d$ and  $n<3K\eta^2$. 
If $W$ is a subset of $V(G)$ with $|W| \le 2D/m^3$, then  $G - W$ contains a $(D, m^4, r)$-adjuster for each $r \le\frac{\eta ^2m^{12}}{20}$.
\end{lemma}

\begin{lemma}\label{path1}
Suppose $t\ge s\ge2$ are integers,  $1/d\ll c\ll 1/K\ll\varepsilon_2\ll\varepsilon_1, 1/s, 1/t$. Let $D=\max\left\{\frac{c^2m^{19}\eta^2}{10^{20}},~\frac{\varepsilon_2\eta^2}{10^{20}}\right\}$ and $G$ be an $n$-vertex $K_{s,t}$-free $(\varepsilon_1, \varepsilon_2\eta^2)$-expander with $\delta(G) \ge d$ and  $n<3K\eta^2$, and let $W$ be a subset of $V(G)$ with $|W| \le 2D/m^3$ and $\ell \le\frac{\eta ^2m^{12}}{20}$.  For each $i \in [2]$, let $U_i \subseteq V (G)-W$ with $|U_i|\ge D$  satisfy $U_1\cap U_2=\emptyset$, $F_i \subseteq G-W -U_1-U_2$ be a $(D, m^4)$-expansion of some $v_i$ with $V(F_1)\cap V(F_2)=\emptyset$. Then there exist $u_i\in U_i$  and a $(u_i,v_i)$-path $P_i$ in $G-W$ for each $i\in[2]$ such that $V(P_1)\cap V(P_2)=\emptyset$ and $\ell\le \ell(P_1) + \ell(P_2) \le \ell + 8m^4$. 
\end{lemma}
\begin{proof}[Proof of Lemma \ref{liu}]
We may assume that $n<3K\eta^2$, or else we can just take $H=G$. By Theorem \ref{turan-kst}, we have $e(G)=nd/2\le t^{1/s}n^{2-1/s}$ and as $\eta=d^{\frac{s}{2(s-1)}}$, it follows that $\eta^2\le (2t^{1/s})^{s/(s-1)}n$. Therefore 
\begin{equation}\label{vertexG}
n\ge\frac{1}{64t}\eta^2.
\end{equation}
This means that 
\begin{equation}\label{4x}
m =\log\frac{15n}{\varepsilon_2\eta^2}\ge \log\frac{15}{64t\varepsilon_2}\ge \max\{50et,  50/\varepsilon_1\},
\end{equation} 
where  the last inequality holds as $\varepsilon_2\ll\varepsilon_1, 1/t$. Let
\begin{equation}\label{maxdegree1}
\Delta:= \text{max}\left\{d/8, cdm ^{20s}\right\}
\end{equation}
and $Z=\{v\in V(G): { d(v)\ge\Delta}\}$.

First, assume that $|Z|\ge \frac{d}{16}\ge 4c\eta$. Recall that $G$ is bipartite.  By the pigeonhole principle, we choose a subset $Z_1\subseteq Z$ with  $2c\eta$ vertices in the same part of  $G$. We claim that for every  $v\in Z_1$, there exists a subset $S(v)\subseteq N(v)$ with $|S(v)|=\Delta/2$ such that $|N(w)\cap Z_1|\le d/\eta$ for every vertex $w\in S(v)$. Indeed, let $A$ be a subset of $N(v)$ such that $|N(u)\cap Z_1|\ge d/\eta+1$ for every $u\in A$. Let $B=Z_1\backslash \{v\}$. Since $G$ is a $K_{s, t}$-free graph, there is no copy of $K_{s-1, t}$ with $t$ vertices in $A$ and $s-1$ vertices in $B$. Hence, by  Corollary \ref{coro2.6} with $\delta=d/\eta$, we have 
\[
\left|Z_1\backslash \{v\}\right|=2c\eta-1\ge\frac{d}{et\eta }|A|^{\frac{1}{s-1}}, 
\]
which together with  \eqref{4x} and \eqref{maxdegree1}  yields that 
\[
|A| \le\left(\frac{\eta et(2c\eta-1)}{d}\right)^{s-1}\le\left(\frac{2c\eta^2et}{d}\right)^{s-1}= d(2cet)^{s-1} \le\Delta/2.
\] 
Thus, we can choose a set $S(v) \subseteq N(v) \backslash A$ with  $|S(v)|=\Delta/2$, as claimed.

Now, we  construct a  $\mathrm{T}K^{(\ell)}_{ c\eta}$ using some vertices in $Z_1$ as  branch vertices. Let   $\mathcal{P}$ be a maximum collection of paths, each of which has a length of $\ell$, satisfying the following rules.
\stepcounter{propcounter}
\begin{enumerate}[label = ({\bfseries \Alph{propcounter}\arabic{enumi}})]
\rm\item\label{rulepath1} Each path $P\in\mathcal{P}$ is a unique $(v_i, v_j)$-path for some $v_i, v_j\in Z_1$.

\rm\item\label{rulepath2} All paths in $\mathcal{P}$ are pairwise internally vertex disjoint, and the internal vertices of those paths are disjoint
from $Z_1$.
\end{enumerate}

In  fact, we find  paths in $\mathcal{P}$ greedily.  In the whole process, a vertex $v\in Z_1$ is called \emph{bad} if there are at least $\Delta/6$ vertices in $S(v)$ used in previous connections and is called \emph{good} otherwise. Let $W_1$ be the set of vertices used in all connections. Then
\begin{equation}\label{w1}
|W_1|\le\tbinom{2c\eta}{2}\cdot (m^{10}+1)\le 4c^2m^{10}\eta^2.
\end{equation}
Recall that for every vertex $v\in Z_1$, we have  $|N(w)\cap Z_1|\le d/\eta$ for every vertex $w\in S(v)$.
By (\ref{4x}) and (\ref{maxdegree1}), the number of bad vertices is at most 
\[
\frac{4c^2m^{10}\eta^2d/\eta}{\Delta/6}\le\frac{24cm^{10}\eta}{m^{20s}}\le c\eta. 
\]
This means that the number of good vertices in $Z_1$ is at least $c\eta$.

\begin{claim}\label{end1}
There exists a path in $\mathcal{P}$ between  every pair of good vertices in $Z_1$.
\end{claim}

\begin{proof}[{Proof of Claim \ref{end1}.}]
By contradiction, suppose that $v_1, v_2 $ are good vertices in $Z_1$ such that no path connects them in $\mathcal{P}$. Recall that $|S(v_i)|=\Delta/2$ for each $i\in[2]$. Let $S_i$ be the subset of $S(v_i)$ consisting of vertices not used in the previous connections. Then  $|S_i|\ge|S(v_i)|-\Delta/6=\Delta/3$. Let $S_i' : = N(S_i)\backslash \{v_i\}$. Note that  $G$ has no copy of $K_{s-1, t}$ with $t$ vertices in $S_i$ and $s-1$ vertices in $S_i'$. By  Corollary \ref{coro2.6} (with $(\delta, A, B)=(d/8-1, S_i, S_i')$), we have
\[
|S_i'|\ge\frac{(\Delta/3)^{\frac{1}{s-1}}(d/8-1)}{et}\ge\frac{d\Delta^{\frac{1}{s-1}}}{48et}.
\]
Then by the choice of $\varepsilon_2$, (\ref{4x}) and (\ref{maxdegree1}), we obtain
\begin{equation}\label{s11}
|S_i'|\ge \max\{c^2\eta^2m^{19}, 4\varepsilon_2\eta^2\}.
\end{equation}

Let $W_2=S_1\cup S_2$ and $W_0=W_1\cup W_2\cup Z_1$. Then  $|W_2|=|S_1\cup S_2|\le \Delta$ and $|W_0 |=|W_1\cup W_2\cup Z_1| \le2c^2m^{10}\eta^2+\Delta+c\eta\le4c^2m^{10}\eta^2 \le D/m^3$ by the choice of $c$. Applying Lemma \ref{adjuster1} with $W = W_0$, we know that $G-W_0$ has a $(D, m^4, 10m^4)$-adjuster, say $\mathcal{A} = (v_3, F_3, v_4, F_4, A)$, where $|A| \le 100m^8$, $\ell(\mathcal{A}) \le |A| + 1 \le 110m^8$ and $|V (F_3)| =|V (F_4)| = D$. Let $\ell'= \ell-10m^4-\ell(\mathcal{A})$. Then $0 \le \ell' \le \frac{m^{12}\eta^2}{20}$.

Note that $|W_0\cup V(F_3)\cup V(F_4)|\le2D+D/m^3$. By \eqref{s11}, we can find a subset $U_i\subseteq S_i'\backslash(W_0\cup V(F_3)\cup V(F_4))$ with $|U_i|\ge2D$ for each $i\in [2]$. So, there exists $U_i'\subseteq U_i$, $|U_i'|=D$ for each $i\in [2]$ and $U_1'\cap U_2'=\emptyset$. Note that $|A \cup W_0 | \le 100m^8 +D/m^3 \le 2D/m^3$.  By Lemma \ref{path1}, we can find a  $(u_1,v_3)$-path $P_1$ and a $(u_2,v_4)$-path $Q_1$ for some $u_i\in U_i'$ with $i\in [2]$ such that $V(P_1)\cap V(Q_1)=\emptyset$, and $\ell' \le \ell(P_1) + \ell(Q_1) \le \ell' + 8m^4$.
Obviously, we can extend $P_1$ to a $(v_1, v_3)$-path $P$ and $Q_1$  to a $(v_2, v_4)$-path $Q$, respectively, such that $V(P)\cap V(Q)=\emptyset$ and $\ell' \le \ell(P) + \ell(Q) \le \ell' + 10m^4$. By the choice of $\ell'$, $\ell(\mathcal{A}) \le \ell- (\ell(P) + \ell(Q)) \le \ell(\mathcal{A}) + 10m^4$. Since $v_1$, $v_2$ belong to the same part and $G$ is bipartite, $\ell(\mathcal{A})$ and $\ell-\ell(P)-\ell(Q)$ are congruential. Thus,  there is a $(v_3, v_4)$-path in $G[A\cup\{v_3, v_4\}]$ of length $\ell - \ell(P)- \ell(Q)$, denoted as $R$, and $P \cup R \cup Q$ is a $(v_1, v_2)$-path of length $\ell$, contradicting to the maximality of $\mathcal{P}$.
\end{proof}
By Claim \ref{end1}, we can find a $(v_1, v_2)$-path in $\mathcal{P}$ for every pair of good vertices $v_1$, $v_2$, which forms a $\mathrm{T}K^ {(\ell)}_{c\eta}$ in $G$.\\

Towards the case when $|Z|< \frac{d}{16}$, we focus on the subgraph $H:= G-Z$. 
We claim that $H$ is an $(\varepsilon_1/2, \varepsilon_2\eta^2)$-expander. Let $X\subseteq V(H)$ with $\varepsilon_2\eta^2/2\le |X|\le |H|/2$. Since $G$ is an $(\varepsilon_1, \varepsilon_2\eta^2)$-expander and $\varepsilon(x) \cdot x$ is increasing when $\varepsilon_2\eta^2/2\le x\le n/2$,  
\[
|N_G(X)| \ge |X| \cdot\varepsilon (|X|, \varepsilon_1, \varepsilon_2\eta^2)\ge \frac{\varepsilon_2\eta^2}{2}\cdot\frac{\varepsilon_1}{\log ^2(15/2)}\ge2|Z|.
\]
This implies that
\[
|N_{H}(X)| \ge|N_G(X)|-|Z|\ge |N_G(X)|/2\ge  |X| \cdot \varepsilon(|X|, \varepsilon_1/2, \varepsilon_2\eta^2), 
\]
and we are done.

It follows from $\delta(G)\ge d/8$ and $|Z|< \frac{d}{16}$ that $\delta(H) \ge \delta(G) - |Z| \ge \frac{d}{16}$. To complete the proof, it suffices to show $|H| \ge  K\eta^2$. We claim that
\begin{equation}\label{6x}
\Delta = cd m^{20s}> d/8. 
\end{equation}
Indeed, if $\Delta = d/8$, then as $\delta(G) \geq d/8$, we have $Z = \{v \in V(G) \mid d(v) \geq \Delta\} = V(G)$. This implies $|Z| \geq d/8$, a contradiction.
 By \eqref{6x} and the choice of $c$, we have 
\begin{equation*}\label{7x}
m^{20s}\ge \frac{1}{8c}\ge \left(\log\frac{30K}{\varepsilon_2}\right)^{20s},
\end{equation*} 
which together with \eqref{4x} yields that $n\ge2K\eta^2$. Thus, $|H|=n-|Z| \ge n/2\ge K\eta^2$.
\end{proof}

\subsubsection{Building adjusters}\label{sub:Building-adjuster}
In the subsection, we prove Lemma \ref{adjuster1}. First, we show that every  $K_{s, t}$-free graph $G$ contains a subgraph with average degree linear to $d(G)$ even after removing a small set of vertices.

\begin{lemma}\label{average-degree}
Suppose $d, t\ge s\ge2$ are integers and $d$ is sufficiently large. There exists $K$ such that the following holds for each $n$ and $d$ satisfying $n< 3K\eta^2$.  Let $G$ be an $n$-vertex $K_{s,t}$-free graph with $\delta(G) \ge d$. Then for any vertex set $W \subseteq V (G)$ with $|W|\leq \min\left\{\frac{\eta^2}{1000t^2},\frac n 2\right\}$, we have $d(G - W) \ge d/4$.
\end{lemma}

\begin{proof}
We can assume   $|W|>3d/4$, otherwise, it is trivial.  Note that 
\[
2e(G-W)+e\left(V(G)\setminus W,W\right)=\sum_{v\in V(G)\setminus W}d_G(v)\ge d(n-|W|),
\]
On the other hand, by Theorem \ref{1996furedi}, we have
\begin{align*}
e\left(V(G)\setminus W,W\right)&\le Z(|W|, n-|W|, s, t)\\
&\le(t-s+1)^{1/s}|W|(n-|W|)^{1-1/s}+(s-1)(n-|W|)^{2-2/s}+(s-2)|W|.
\end{align*}
Therefore,
\begin{align*}
2e(G-W)\ge d(n-|W|)-(t-s+1)^{1/s}|W|(n-|W|)^{1- 1/s}-(s-1)(n-|W|)^{2-2/s}-(s-2)|W|. 
\end{align*}
So, the average degree of $G-W$ satisfies that  
\begin{align}\label{average-G-W}
d(G-W)\ge d-(t-s+1)^{1/s}|W|(n-|W|)^{- 1/s}-(s-1)(n-|W|)^{1-2/s}-(s-2)|W|(n-|W|)^{-1}. 
\end{align}

By Theorem \ref{turan-kst}, we have $n\ge \frac{1}{64t}\eta^2.$ Note that the right hand of \eqref{average-G-W} is a decreasing function of $|W|$. Using the fact   $|W|\leq \min\left\{\frac{\eta^2}{1000t^2},\frac n 2\right\}$ and $n<3K\eta^2$, we can deduce that 
\begin{equation}\label{equ:ave}
\begin{aligned}
d(G-W) &\ge d - (t-s+1)^{1/s}|W|(n-|W|)^{-1/s} - (s-1)nd^{1-2/s} - (s-2)|W|(n-|W|)^{-1} \\
&\ge d -(t-s+1)^{1/s} \frac{\eta^2}{1000t^2}\left(\frac{3K\eta^2}{2}\right)^{-1/s}-3(s-1)K\eta^2d^{1-2/s}-(s-2) \frac{\eta^2}{1000t^2}\left(\frac{3K\eta^2}{2}\right)^{-1}\\
&\ge d - c_1d - c_2d^{\frac{s-2}{s-1}} - (s-2)\cdot\frac{2}{1000t^2\cdot3K}\\
\end{aligned}
\end{equation}
where 
\[
c_1=\frac{2(t-s+1)^{1/s}\frac{1}{1000t^2}}{\left(\frac{1}{64t}\right)^{1/s}}<\frac{128}{1000} \quad \quad\text{ and } \quad \quad c_2=(s-1)(3K)^{1-2/s}. 
\]

Since $c_1 < \frac{128}{1000}$, we conclude that
\[
d(G-W) \ge d - c_1 d - c_2 d^{1-2/s} - O(1) \ge d/4
\]
holds for all sufficiently large $d$.
\end{proof}

Now we prove Lemma \ref{adjuster1} by induction on $r$. 

\begin{proof}[{Proof of Lemma \ref{adjuster1}}]  
Let $G$ be an $n$-vertex $K_{s,t}$-free  $(\varepsilon_1, \varepsilon_2\eta^2)$-expander with $\delta(G) \ge d$ and  $n<3K\eta^2$. First, we claim that for any $W'\subseteq V(G)$ with $|W'| \le 4D$,  $G - W'$ contains a $(D, m^4, 1)$-adjuster. It follows from $n<3K\eta^2$ that $m\le \log (45K/\varepsilon_2)$. 

Since $|W'| \le 4D$ and $D = \max\left\{\frac{c^2 m^{19}\eta^2}{10^{20}},\frac{\varepsilon_2\eta^2}{10^{20}}\right\}$, we have $|W'| \le \max\left\{\frac{4c^2 m^{19}\eta^2}{10^{20}},\frac{4\varepsilon_2\eta^2}{10^{20}}\right\}$. By the choice $1/d\ll c\ll 1/K\ll\varepsilon_2\ll\varepsilon_1, 1/s, 1/t$, it holds that $|W'| \le \min\left\{\frac{\eta^2}{1000t^2},\frac{n}{2}\right\}$. Therefore, by Lemma \ref{average-degree},  $d(G-W') \ge \frac{d}{4}$. By Theorem \ref{subgraph-expander}, there exists a bipartite $(\varepsilon_1, \varepsilon_2\eta^2)$-expander $H= G-W'$ with $\delta(H) \ge \frac{d}{32}$. Using  Lemma \ref{diamter}, we can find a cycle $C$ in $H$ of length $2r$ such that $2r\le \frac{m^4}{32}$. Pick two vertices $v_1$, $v_2 \in V (C)$ of distance $r-1$ in  $C$. For each $i\in [2]$, using the minimum degree condition of $H$, we can find a subset $A_i\subset N_{H-C}(v_i)$ of size $\frac{d}{600}$ such that $A_1\cap A_2=\emptyset$. Let $B_1=N_{H-C}(A_1)\backslash  A_2$ and $B_2=N_{H-C}(A_2)\backslash  A_1$. Note that $G[A_i, B_i]$ does not contain a copy of $K_{s-1, t}$ with $t$ vertices in $A_i$ and $s - 1$ vertices in $B_i$. For each $v\in A_i$, there are at least $\frac{1}{32}d-|C|-|A_1|-|A_2|\ge \frac{1}{64}d$ neighbors in $B_i$. Now we can use Corollary \ref{coro2.6} (with $\delta: =\frac{d}{64}$) to get that 
\[
|B_i|\ge\frac{1}{et}\left(\frac{d}{64}\right)\left(\frac{d}{600}\right)^{\frac{1}{s-1}}\ge 2D.
\]
For each $i\in[2]$, we can choose subsets $B_i'\subset B_i$ with $|B_i'|=D$ such that $B'_1\cap B'_2=\emptyset$.  Thus, we  find two vertex-disjoint $(D, 2)$-expansion of $v_1$ and $v_2$ respectively,  say $F_1$, $F_2$. We conclude that $v_1, v_2, F_1, F_2, C$ forms  a $(D, m^4, 1)$-adjuster.

Now we complete our proof by induction on $r$. It is true when $r=1$ by the above arguments. For $1 < r< \frac{\eta ^2m^{12}}{20}$, suppose that there exists a  $(D, m^4, r)$-adjuster in $G-W$, say  $\mathcal{A}_1 = (v_1, F_1, v_2, F_2, A_1)$.  Let $W_1 = W \cup A_1\cup V (F_1)\cup V (F_2)$. Then $|W_1| \le 4D$. By the above claim, $G - W_1$ contains a $(D, m^4, 1)$-adjuster, say $\mathcal{A}_2= (v_3, F_3, v_4, F_4, A_2)$. We want to use Lemma \ref{diamter} to find a path $P$ connecting $V (F_1)\cup V (F_2)$ and $V (F_3)\cup V (F_4)$ avoiding $W \cup A_1\cup A_2$. Note that $|V(F_1)\cup V(F_2)| = |V(F_3)\cup V(F_4)|= 2D$, and $|W \cup A_1\cup A_2| \le 2D/m^3  + 10rm^4+ 10rm^4 \le 2D/m^3  + 20rm^4  \le 3D/m^3$ by the fact $r< \frac{\eta ^2m^{12}}{20}$. Recall that $\varepsilon(2D):=\varepsilon(2D,\varepsilon_1,k)$is decreasing when $2D\ge k$, where $\varepsilon$ is as in Definition \ref{subk}. Then we have
\begin{equation*}
\varepsilon(2D)\cdot 2D/4\ge  \varepsilon(n)\cdot 2D/4=\varepsilon_1/m^2\cdot 2D/4\ge 3D/m^3\ge |W \cup A_1 \cup A_2|.
\end{equation*}
Thus, by Lemma \ref{diamter}, there exists a path $P$ of length at most $m^4$ from $V (F_1)\cup V (F_2)$ to $V (F_3)\cup V (F_4)$ avoiding $W \cup A_1\cup A_2$. Without loss of generality, we may assume that $P$ is a path from $V (F_1)$ to $V (F_3)$. It follows from $F_1$ and $F_3$ are $(D, m^4)$-expansions of $v_1$ and $v_3$, respectively,  that  $P$ can be extended to a $(v_1, v_3)$-path, denoted by $Q$, of length at most $3m^4$ via the interiors of $F_1$ and $F_3$.

We claim that $(v_2, F_2, v_4, F_4, A_1 \cup A_2 \cup V (Q))$ is a $(D, m^4, r+ 1)$-adjuster. It is easy to check that \ref{Ad1} and \ref{Ad2} hold. For \ref{Ad3}, we have $V(F_2)\cap V(F_4)\cap (A_1 \cup A_2 \cup V (Q))=\emptyset$ and $|A_1\cup A_2\cup V (Q)| \le 10rm^4+10rm^4+3m^4 \le 10(r+1)m^4$. For \ref{Ad4}, let $\ell = \ell(\mathcal{A}_1)+\ell(\mathcal{A}_2)+\ell(Q)$. For every $i \in \{0, 1, \ldots, r + 1\}$, there is some $i_1 \in \{0, 1, \ldots, r\}$ and $i_2 \in \{0, 1\}$ such that $i = i_1 + i_2$. Let $P_1$ be a $(v_1, v_2)$-path in $G[A_1 \cup \{v_1, v_2\}]$ of length $\ell(A_1) + 2i_1$ and  $P_2$ be a $(v_3, v_4)$-path in $G[A_2 \cup \{v_3, v_4\}]$ of length $\ell(A_2) + 2i_2$. Hence, $P_1 \cup Q \cup P_2$ is a $(v_2, v_4)$-path in $G[A_1 \cup A_2 \cup V (Q)]$ of length $\ell + 2i$. This completes the proof of Lemma \ref{adjuster1}.
\end{proof}

\subsection{Proof of Lemma \ref{yang}}\label{sec5}

To construct a balanced subdivision as in Lemma~\ref{yang}, our approach follows the strategy of Luan, Tang, Wang and Yang \cite{2023yang} which gives a more efficient construction of reasonable-sized adjusters (see Lemma \ref{adjuster2}) by concatenating many disjoint small adjusters. Our construction mainly consists of two steps: (1) build a sufficient number of units whose interiors are pairwise disjoint (see Lemma \ref{unit}); (2) for every pair of units, connect their core vertices, one by one, by a unique path of a fixed length $\ell$ whilst avoiding previous connections. To achieve (2), we use adjusters: we first build an adjuster which is disjoint from previous connections and the interiors of all units (use Lemma \ref{adjuster2}), and then join the adjuster to the exteriors of the two units via two disjoint paths so that the sum of their lengths is close to $\ell$ (see Lemma~\ref{path2}). Finally, we use the property of the adjuster to obtain a desired connection.

The following lemma shows that we can find a desired unit after deleting a small number of vertices in a $K_{s,t}$-free  bipartite expander. The proof is similar as that of Lemma 3.8 in \cite{2023yang}, so we include it in Appendix \ref{unit-le}.

\begin{lemma}\label{unit}
Suppose $1/n, 1/d\ll c_1\ll1/K\ll \varepsilon_1, \varepsilon_2, 1/s, 1/t$, and $t, s\in \mathbb{N}$ satisfy $t\ge s\ge2$ and $d\ge \log^{200}n$. Let $G$ be an $n$-vertex $K_{s,t}$-free  bipartite $(\varepsilon_1, \varepsilon_2\eta^2)$-expander with $\delta(G) \ge d$ and $n \ge K\eta^2$. If  $W$ is a subset of $V (G)$ with $|W|\le4c_1\eta^2m^{20}$, then  $G-W$ contains a $(c_1\eta, m^{20}, c\eta, 2m^4)$-unit.
\end{lemma}

Lemma \ref{adjuster2} provides us with a desired adjuster, whose proof can be found in   Appendix \ref{3.12-le}.
\begin{lemma}\label{adjuster2}
Suppose $1/n, 1/d\ll 1/K\ll \varepsilon_1, \varepsilon_2, 1/s, 1/t$ and $t, s\in \mathbb{N}$ satisfies $t\ge s\ge2$, and $d\ge \log^{200}n$. Let $D=\frac{\eta^2m^{20}}{10^{10}t}$, and $G$ be an $n$-vertex $K_{s,t}$-free  bipartite $(\varepsilon_1, \varepsilon_2\eta^2)$-expander with $\delta(G) \ge d$ and $n \ge K\eta^2$.  For any set  $W\subseteq V(G)$ with $|W| \le \frac{D}{ \log^3n/\eta^2}$, we get that $G-W$ contains a $(D, m^4, r)$-adjuster for any positive integer $r$ with $r \le\frac{\eta^2m^{12}}{20}$.
\end{lemma}
The following lemma helps us to connect two pairs of vertex sets even after deleting some vertices. We leave its proof in Appendix \ref{3.7-le}.
\begin{lemma}\label{path2}
Suppose $1/n, 1/d\ll 1/K\ll \varepsilon_1, \varepsilon_2, 1/s, 1/t$, and $t, s\in \mathbb{N}$ satisfies $t\ge s\ge2$ and $d\ge \log^{200}n$. Let $D=\frac{\eta^2m^{20}}{10^{10}t}$, $\ell\le\frac{\eta^2m^{12}}{20}$ and $G$ be an $n$-vertex $K_{s,t}$-free bipartite $(\varepsilon_1, \varepsilon_2\eta^2)$-expander with $\delta(G) \ge d$ and $n\ge K\eta^2$. Let $W$ be a subset of $V(G)$ with $|W| \le \frac{D}{ \log^3 n/\eta^2}$, $U_i \subseteq V (G)-W$ be disjoint vertex sets of size at least $D$ for each $i \in [2]$, and $F_j \subseteq G-W -U_1-U_2$ be vertex-disjoint $(D, m^4)$-expansion of $v_j$ for each  $j\in\{3, 4\}$. Then, $G-W$ contains vertex-disjoint paths  $P_1$ and $P_2$ with $\ell\le \ell(P_1) + \ell(P_2) \le \ell + 10m^4$ such that both $P_1$ and $P_2$ connect $\{v_1, v_2\}$ to $\{v_3, v_4\}$ for some $v_i \in U_i$ with $i \in [2]$.
\end{lemma}

\begin{figure}[h]
\begin{center}
\scalebox {1.00}[1.00]{\includegraphics {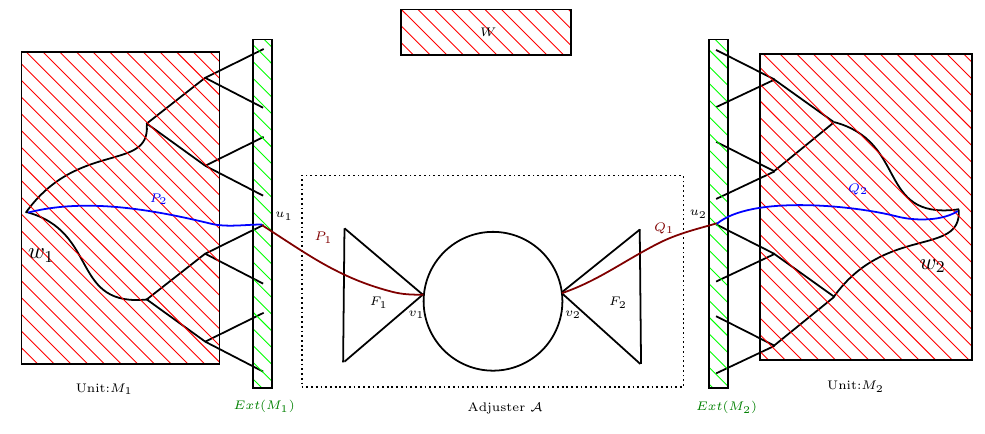}}
\end{center}
\caption{An illustration of the proof of Lemma \ref{yang}.}
\end{figure}
\begin{proof}[{Proof of Lemma \ref{yang}.}]
Let $G$ be an $n$-vertex  bipartite  $K_{s,t}$-free $(\varepsilon_1, \varepsilon_2\eta^2)$-expander with $\delta(G) \ge \frac{d}{16}$. Recall that $\ell$ is the smallest even integer larger than $m^{10}$ and let $c_1$ be such that the condition in Lemma \ref{unit} applies. Using Lemma \ref{unit}, we can find a collection $\{M_1, \ldots, M_{2c_1'\eta}\}$ of pairwise interior-disjoint $(c_1'\eta, m^{20}, c_1'\eta, 2m^4)$-units in $G$ such that $c_1'\ll c_1$, $M_i\cap M_j=\emptyset$ for $i\neq j$ and suppose that each $M_i$ has the core vertex $w_i$. Indeed, suppose that we have found $M_1,\ldots, M_i$ one by one  for $i<2c_1'\eta$.  Let $W'$ be the union of the interiors of those units that we have found. Then $|W'|< 2c_1' \eta(2m^4 + 1 + m^{20}) \cdot c_1' \eta \le 4c_1'^2\eta^2m^{20}$. Applying Lemma \ref{unit} on $G$ (with $W:=W'$) yields another $(c_1'\eta, m^{20}, c_1' \eta, 2m^4)$-unit $M_{i+1}$ disjoint from $W'$. Note that $G$ is bipartite. By the pigeonhole principle, we can find $c_1' \eta$ such units among them such that their core vertices are in the same part of the bipartition for $G$. Without loss of generality, these units are $M_1, \ldots, M_{c_1'\eta }$.  We denote by $u_{i, j}$ the center of the $j$-th hub in $M_i$ for  $1 \le i, j \le c_1' \eta$. Let $W$ be the union of the vertices on the $(w_i, u_{i, j})$-paths in all units. Then $|W| \le c_1' \eta \cdot (2m^4 + 1) \cdot c_1' \eta \le 4c_1'^2\eta^2m^4$.

Using the similar method as in  the proof of Lemma \ref{liu} with a careful analysis, we construct a $\mathrm{T}K^{(\ell)}_{c_1'\eta/2}$  using vertices in $\{w_1, \ldots, w_{c_1'\eta }\}$ as core vertices. Let   $\mathcal{Q}$ be a maximal collection of pairwise vertex-disjoint paths in $G-W$ satisfying the following.
\stepcounter{propcounter}
\begin{enumerate}[label = ({\bfseries \Alph{propcounter}\arabic{enumi}})]
\rm\item\label{rulepathQ1} For $i,j\in[c_1'\eta]$, the path $Q_{i,j}\in\mathcal{Q}$ is  a  $(w_i,w_j)$-path of length $\ell$, which is an extension of some $(\mathsf{Ext}(M_i)$, $\mathsf{Ext}(M_j))$-path $P_{i,j}$. 

\rm\item\label{rulepathQ2} For each pair of units, there is at most one path in $\mathcal{Q}$.
\end{enumerate}
Let us call a unit \emph{bad} if there are more than $\eta m^{16}$ vertices from its interior used in paths in $\mathcal{Q}$, and \emph{good} otherwise. Let $W_1$ be the set of vertices of $\mathcal{Q}$. Then by \ref{rulepathQ1} and \ref{rulepathQ2}, 
\[
|W_1| \le  \tbinom{c_1'\eta}{2}(m^{10}+1)\le  40c_1'^2\eta^2m^{10}.
\] 
This means that there are most  $\frac{40c_1'^2\eta^2m^{10}}{\eta m^{10}}\le 40c_1'^2\eta \le c_1'\eta /2$ bad units. Thus, the number of good units is at least  $c_1'\eta /2$. 

To complete the proof, it suffices to show that for every pair of good units, there is a path in  $\mathcal{Q}$ connecting their core  vertices. Indeed, suppose for a contradiction that there is no desired path in $\mathcal{Q}$ for good units $M_1$ and $M_2$. For $i\in [2]$, let $I_i$ denote the index set such that for each $k\in I_i$ the center  of the hub $H(u_{i, k})$ is not used in $\mathcal{Q}$. By \ref{rulepathQ2}, we have  $|I_i| \ge c_1'\eta/2$.  Let $A_i$ be the set of vertices in $\bigcup _{k\in I_i}S_i(u_{i, k})$ not used in $\mathcal{Q}$. It follows from  $M_i$ is a good unit that
\[
|A_i| \ge \bigcup _{k\in I_i}S_i(u_{i,k}) - \eta m^{16} \ge c_1'\eta m^{20}/2-\eta m^{16} \ge c_1'\eta m^{20}/4.
 \]

Recall that  hubs in $M_i$ are vertex disjoint, which implies
\[
|N_{M_i}(A_i)\backslash W| \ge c_1'^2\eta^2 m^{20}/4 - 4c_1'^2\eta^2m^4 \ge c_1'^2\eta^2m^{20}/8.
 \]
Let $B=\mathsf{Int}(M_1)\cup \mathsf{Int}(M_2)$ and $W_0=W\cup W_1\cup B$. Then 
\[
|B| \le 2 \cdot 2c_1 \eta m^{20}
\]
and 
\[
|W_0|\le 4c_1'^2\eta^2m^4+40c_1'^2\eta^2m^{10}+4c_1' \eta m^{20}\le 2500c_1'^2\eta^2m^{10} \le \frac{D}{2 \log^3n/\eta^2}
\]
as $d \ge \log^{200} n \ge m^{200}$.
Using Lemma \ref{adjuster2} (with $W: = W_0$) yields that $G-W_0$ has a $(D, m^4, 16m^4)$-adjuster $\mathcal{A}= (v_1, F_1, v_2, F_2, A)$  with  $|A| \le 160m^8$, $\ell(\mathcal{A}) \le |A| + 1 \le 170m^8$ and $|V (F_1)| = |V (F_2)| = D$. Let $\ell' = \ell - 16m^4 - \ell(\mathcal{A})$ for $i\in [2]$. Then $0 \le \ell' \le\eta^2m^{12}$. For $i\in [2]$, let $U_i =N_{M_i}(A_i)\backslash(W_0  \cup V (F_1) \cup V (F_2))$. Then $|U_i|\ge c_1'^2\eta^2m^{20}/8-\frac{D}{2\log^3n/\eta^2} - 2D \ge 2D$. So, there exist disjoint vertex sets $U_i'\subseteq U_i$ such that $|U_i'| \ge D$. As $d \ge \log^{200} n \ge m^{200}$, we have $|A \cup W_0 | \le 160m^8 + \frac{D}{2 \log^3n/\eta^2} \le \frac{D}{\log^3n/\eta^2}$. By Lemma \ref{path2}, there is a $(u_1, v_1)$-path
$P_1$ and a $(u_2, v_2)$-path $Q_1$ for $u_1 \in U_1'$ and $u_2 \in U_2'$ in $G - A - W_0$ such that $V(P_1)\cap V(Q_1)=\emptyset$ and $\ell \le\ell(P_1) + \ell(Q_1) \le \ell + 10m^4$. As $U_1'\subset \mathsf{Ext}(M_1) $, there exists a $(w_1, u_1)$-path $P_2$ of length at most $2m^4 + 2 \le 3m^4$ in $M_1$. Similarly, we can find a $(w_2, u_2)$-path $Q_2$ of length at most $2m^4 + 2 \le 3m^4$ in $M_2$. Let $P = P_1 \cup P_2$ and $Q = Q_1 \cup Q_2$. Then $P$ is a $(w_1, v_1)$-path and $Q$ is a $(w_2, v_2)$-path  with $\ell'\le \ell(P) + \ell(Q) \le \ell' + 16m^4$. Consequently, $\ell(\mathcal{A}) \le \ell' - \ell(P) - \ell(Q) \le \ell(\mathcal{A}) + 16m^4$. As $G$ is a bipartite graph and $w_1$, $w_2$ are in the same part, $\ell(\mathcal{A})$ and $\ell -\ell(P)-\ell(Q)$ are congruential. There is a $(v_1, v_2)$-path $R$ in $G[A\cup\{v_1, v_2\}]$ of length $\ell - \ell(P) -\ell (Q)$. Then $P \cup R \cup Q$ is a $(w_1, w_2)$-path of length $\ell$ which satisfies \ref{rulepathQ1}-\ref{rulepathQ2}, contradicting to the maximality of $\mathcal{Q}$. This completes the proof. 
\end{proof}

\section{Proof of Theorem \ref{main-thm-2}}\label{sec6}
In this section, we  prove  Theorem \ref{main-thm-2}. Now we  introduce the following two lemmas  given by Sudakov and Verstra\"{e}te \cite{sud}.

\begin{lemma}\label{densecase1}
{\rm(Lemma 3.1  in \cite{sud})}  
Let $\mathcal{P}$ be a monotone property (closed under taking subgraphs) of graphs, and suppose that for every graph $G\in \mathcal{P}$ with minimum degree $d$, and every set $X \subset V (G)$ of size at most $f(d)$,
\[
|N(X)| > 2|X|.
\]
Then every $G \in \mathcal{P}$ of average degree at least $16d$ contains cycles of $3f(d)$ consecutive even lengths, the shortest having length at most twice the largest radius of any component of $G$.
\end{lemma}

\begin{lemma}\label{densecase2}
{\rm(Lemma 3.2  in \cite{sud})} Let $a > 0$, $1/2 <b< 1$ be reals such that for any positive integer $n$, $ex(n, H) \le an^{2b}$. Then, for any $H$-free graph $G$ of minimum degree at least $18ad$, and any subset $X$ of vertices of $G$ of size at most $d^{\frac{1}{2b-1}}$, $|N(X)|>2|X|$.
\end{lemma}

\begin{definition}
\rm
(\cite{2023liu})
For any connected bipartite graph $H$ and $u,v\in V(H)$, let
$$
\pi(u, v, H)=
    \begin{cases}
        0, &{\rm if}\ u=v.\\
        1, &{\rm if}\ u\ {\rm and}\ v\ {\rm are\ in\ the\ different \ vertex\ classes\ in\ the\ (unique)\ bipartition\ of}\ H.\\
        2, &{\rm if}\ u\ {\rm and}\ v\ {\rm are\ in\ the\ same\ vertex\ class\ and}\ u\neq v. 
    \end{cases}
    $$
\end{definition}

To prove Theorem \ref{main-thm-2}, we divide our proof into three cases   according to the average degree $d$ of an $n$-vertex $K_{s, t}$-free bipartite expander. For the sparse case $d\le\log^{200}n$, we can do better through the following theorem. 

\begin{theorem}\label{liusparse}
There exists $\varepsilon_1 > 0$, such that, for each $0 < \varepsilon_2 < 1/5$ and integers $t\ge s\ge2$, there exists $d_0 = d_0(\varepsilon_1, \varepsilon_2)$ such that the following holds for each $n \ge d \ge d_0$. Suppose that $H$ is an $n$-vertex $K_{s, t}$-free bipartite  $(\varepsilon_1, \varepsilon_2\eta^2)$-expander with $\delta(H) \ge d$. Let $x$, $y \in V (H)$ be distinct, and let $\ell\in[\log^7 n, n/ \log^{12} n]$ satisfy $\pi(x, y, H) \equiv \ell (\mod\: 2)$.
Then $H$ contains an $(x,y)$-path with length $\ell$.
\end{theorem}

We remark that Theorem \ref{liusparse} is a slight variation of Theorem 2.7 due to Liu and Montgomery \cite{2023liu}. We omit its proof and interested readers can find it in Appendix \ref{4.3-le}.

\begin{proof}
[Proof of Theorem \ref{main-thm-2}.] Let $\varepsilon_1$ be such that the condition in Theorem \ref{liusparse} applies, and we choose $1/n, 1/d\ll c_1\ll1/K\ll \varepsilon_1, \varepsilon_2, 1/s, 1/t$,  $t, s\in \mathbb{N}$ satisfying $t\ge s\ge2$. Let $G$ be an $n$-vertex  $K_{s, t}$-free  bipartite $(\varepsilon_1, \varepsilon_2\eta^2)$-expander with average degree $d$ and $\varepsilon=K^{-\frac{s-1}{s}}$. The case that $d\le\log^{200}n$ is obvious by Theorem \ref{liusparse}. For the case  $d\ge \varepsilon n^{(s-1)/s}$, let $d_1=d/(288t^{1/s})$. Then by Theorem \ref{turan-kst} and Lemma \ref{densecase2} (with $(H, a, b)=(K_{s,t}, t^{1/s}, 1-\frac{1}{2s})$), for each $X\subseteq V(G)$ of size  at most $d_1^{s/(s-1)}$, we have $|N(X)|>2|X|$. Let $f(d)=d_1^{s/(s-1)}$. Using Lemma \ref{densecase1} with $\mathcal{P}$ being the family of all $K_{s,t}$-free graphs, we obtain   that $G$ has $3f(d_1)$ cycles of consecutive even lengths, and the length of shortest cycle is at most twice than that of radius of $G$.  On the other hand, by  Lemma  \ref{diamter},  the radius of $G$ is at most 
$\frac{2}{\varepsilon_1}\log^3\left(\frac{15n}{\varepsilon_2\eta^2}\right)$. Thus, 
 $\mathcal{C}(G)$ contains every even integer in $\left[\frac{4}{\varepsilon_1}\log^3\left(\frac{15n}{\varepsilon_2\eta^2}\right),  \frac{\eta^2}{(288t^{1/s})^{s/(s-1)}}\right]$. 

Suppose that $\log^{200}n\le d\le \varepsilon n^{\frac{s-1}{s}}$. By Lemma \ref{unit}, we can  find two $(c_1\eta, m^{20}, c_1\eta,$ $2m^4)$-units $M_1$ and $M_2$  with core vertices $v_1, v_2$, whose interiors are pairwise disjoint. Note that $|\mathsf{Ext}(M_i)|\ge c_1^2\eta^2m^{20}$ for each $i\in[2]$ and $|\mathsf{Int}(M_1)\cup \mathsf{Int}(M_2)|\le 2 \cdot 2c_1 \eta m^{20}$. By Lemma \ref{diamter}, there is a path $P_0$  of length at most $m^4$ from some $v_1'\in \mathsf{Ext}(M_1)$ to some $v_2'\in \mathsf{Ext}(M_2)$ while avoiding the vertices in $\mathsf{Int}(M_1)\cup \mathsf{Int}(M_2)$.  Then we can extend $P_0$ to a $(v_1, v_2)$-path $P_1$ of length at most $m^4+ 2m^4+ 2+2m^4+2 \le6m^4$ using two subpaths in the units $M_1$ and $M_2$. 

Let $\ell$ be an even integer in $\left[200m^8, \frac{\eta^2m^{12}}{20}\right]$. In the following, we will find a $(v_1, v_2)$-path $P$ with length  $\ell$   while avoiding the interiors of $P_1$. This implies that $P\cup P_1$ is a cycle of length at most $\ell+6m^4$. Consequently, $\mathcal{C}(G)$ contains every even integer in $\left[300m^8, \frac{\eta^2m^{12}}{100}\right]$, which completes the proof.  Let $D=\max\left\{\frac{c^2m^{19}\eta^2}{10^{20}}, \frac{\varepsilon_2\eta^2}{10^{20}}\right\}$, and  $\ell'=\ell-16m^4 -\ell(\mathcal{A})$. Then $m^8\le \ell'\le \frac{\eta^2m^{12}}{20}$. Using Lemma \ref{adjuster2}  (with $W=V(P_1)\cup(\mathsf{Int}(M_1)\cup \mathsf{Int}(M_2)) $), we can find a $(D, m^4, 16m^4)$-adjuster in $G-(V(P_1)\cup\mathsf{Int}(M_1)\cup \mathsf{Int}(M_2))$, say $\mathcal{A}=(v_3, v_4, F_3, F_4, A)$. Note that  $|V(F_3)|=|V(F_4)|=D$, $|A|\le160m^8$ and $\ell(\mathcal{A})\le|A|+1\le170m^8$. For $i\in [2]$, let $U_i=\mathsf{Ext}(M_i)\backslash (V(P_1)\cup V(F_3)\cup V(F_4))$. Then $|U_i|\ge c_1^2\eta^2m^{20}-6m^4-2D\ge  c_1^2\eta^2m^{20}/2\ge2D$.  So there exists $U_{i}'\subseteq U_i$ with $|U_{i}'|=D$  and   $U_{1}'\cap U_{2}'=\emptyset$. Note that  $|A \cup \mathsf{Int}(M_1)\cup \mathsf{Int}(M_2)| \le D/\log^3\frac{n}{\eta^2}$. Then,  by Lemma \ref{path2}, we can find a $(u_1,v_3)$-path $P_2$ and a $(u_2,v_4)$-path $Q_2$ for some $u_i' \in U_{i}$ such that $V(P_2)\cap V(Q_2)=\emptyset$,  and $\ell' \le \ell(P_2) + \ell(Q_2) \le \ell' + 10m^4$. We can extend $P_2$ to a $(v_1, v_3)$-path $P_3$ through the vertices in $M_1$ and $Q_2$  to a $(v_2, v_4)$-path $Q_3$ through the vertices in $M_2$ with  $\ell' \le \ell(P_3) + \ell(Q_3) \le \ell' + 16m^4$. Now, $\ell(\mathcal{A}) \le \ell- (\ell(P_3) + \ell(Q_3)) \le \ell(\mathcal{A}) + 16m^4$. As $G$ is a bipartite graph, $\ell(\mathcal{A})$ and $\ell-\ell(P)-\ell(Q)$ are congruential. Thus we can find a  $(v_3, v_4)$-path $R$ in $G[A\cup\{v_3, v_4\}]$ of length $\ell - \ell(P_3)- \ell(Q_3)$, which implies that $P_3 \cup R \cup Q_3$ is a $(v_1, v_2)$-path with length  $\ell$.
\end{proof}

\section{Concluding remarks}
In this paper, our main result strengthens previous results by demonstrating that the  condition excluding $K_{s,t}$ not only forces large clique subdivisions, but also guarantees such subdivisions in a balanced form. 
More generally, the $K_{s,t}$-free condition appears to be a natural condition forcing significantly larger balanced clique subdivisions than those obtained in general  graphs under the same average degree condition. It would be interesting to investigate whether Theorem~\ref{main-thm} can be extended beyond forbidding complete bipartite subgraphs, for instance, to graphs excluding more general bipartite configurations. 
In a recent breakthrough result, Liu and Montgomery~\cite{2023liu} gave the first construction of even cycles with precisely controlled lengths under the average degree condition. Along the direction, our second result improves their work in the setting of $K_{s,t}$-free graphs, showing that forbidding complete bipartite graphs leads to substantially stronger conclusions. 
A natural remaining challenge is to find many odd cycles of prescribed lengths in non-bipartite $K_{s,t}$-free graphs with high average degree (or chromatic number) restriction.

\section*{Acknowledgement}
We thank the anonymous referees for their careful reading and valuable comments which greatly improve the presentation of our manuscript. The second author would like to thank Xizhi Liu for fruitful discussions.

\small {

}

\appendix

\section{Proofs of Lemmas \ref{path1} and \ref{path2}}\label{3.7-le}
In this section, we mainly complete the proof of Lemma \ref{path1}. The proof of Lemma \ref{path2} is similar to the proof Lemma \ref{path1}, the sight difference is when prove Lemma \ref{path2}, we use a unit (whose existence  can be guaranteed by Lemma \ref{unit}) to substituted the bipartite graph $G[A,B]$ in the proof of Lemma \ref{path11}. 

\begin{lemma}\label{path11}
Suppose $t\ge s\ge2$ are integers,  $ 1/d\ll c,~\varepsilon_2\ll1/\log K,~\varepsilon_1, 1/s,  1/t$.  Let $D=\max\left\{\frac{c^2m^{19}\eta^2}{10^{20}}, \frac{\varepsilon_2\eta^2}{10^{20}}\right\}$, and $G$ be an $n$-vertex $K_{s,t}$-free  $(\varepsilon_1, \varepsilon_2\eta^2)$-expander with $\delta(G) \ge d$ and $n<3K\eta^2$. Suppose $F_i$ is a $(D, m^4)$-expansion of $v_i$ in $G$ for each $i\in [2]$, and $W \subseteq V (G)\backslash(V (F_1) \cup V (F_2))$ satisfies $|W| \le 2D/m^3 $. Then for any $\ell\le \frac{\eta^2m^{12}}{20}$, there is a $(v_1, v_2)$-path in $G - W$ with length between $\ell$ and $\ell+9m^4$.
\end{lemma}
\begin{proof}
Let $(P_1, P_2, v_3, v_4, F_3, F_4)$ be a six-tuple satisfying the following properties.
\stepcounter{propcounter}
\begin{enumerate}[label = ({\bfseries \Alph{propcounter}\arabic{enumi}})]
\rm\item\label{AppenrulepathsP1} For each $i\in [2]$, $P_i$ is a $(v_i, v_{i+2})$-path in $G - W$ and $\ell(P_1)+\ell(P_2)$ is at most $\ell + 4m^4$.

\rm\item\label{AppenrulepathsP2} For each $i\in \{3,4\}$, $F_i$ is a $(D, 2m^4)$-expansion of $v_i$ in $G - W$ with $V (F_i) \cap V (P_{i-2}) = \{v_i\}$, and $V (P_1 \cup F_3) \cap V (P_2 \cup F_4)=\emptyset$.

\rm\item\label{AppenrulepathsP3}   Subject to \ref{AppenrulepathsP1} and \ref{AppenrulepathsP2}, $\ell(P_1)+\ell(P_2)$ is maximum. 
\end{enumerate}

Note that such a sex-tuple exists as $(v_3, v_4, v_1, v_2, F_1, F_2)$ for any $v_3\in V(F_1), v_4\in V(F_2)$ satisfying \ref{AppenrulepathsP1} and \ref{AppenrulepathsP2}.

We claim that $\ell(P_1)+\ell(P_2) \ge \ell$. Suppose to the contrary that $\ell(P_1)+\ell(P_2) < \ell$. Let $W'=W \cup V (P_1\cup P_2) \cup V (F_3\cup F_4)$. Then $|W'| \le 2D/m^3 +2(\ell+4m^4) +2D\le 4D$. It follows from  Lemma \ref{average-degree} that $d(G-W') \ge d/4$, which implies that there exists a subgraph $H'\subseteq G-W'$ such that  $\delta(H')\ge d/8$. Choose a vertex  $v\in V(H')$ and a set $A\subseteq N_{H'}(v)$ with $|A|=\frac{d}{600}$ and let $B=N_{H'}(A)\backslash\{v\}$. Note that $H'[A, B]$ does not contain a copy of $K_{s-1, t}$ with $t$ vertices in $A$ and $s - 1$ vertices in $B$. Using Corollary \ref{coro2.6} yields that
\[
|B|\ge\frac{1}{et}\left(d/8-1\right)\left(\frac{d}{600}\right)^{\frac{1}{s-1}}\ge D.
\]
Since $|W \cup V (P_1) \cup V (P_2)| \le 3D/m^3 $, we have $ \varepsilon(D)\cdot D/4 \ge   \varepsilon(n)\cdot D/4 \ge 3D/m^3 \ge |W \cup V (P_1\cup P_2)|$. Then by Lemma \ref{diamter}, there is a path $Q'$ from $V (B )$ to $V (F_3)\cup V(F_4)$ of length at most $m^4$, avoiding $(W \cup V (P_1)\cup V (P_2))\backslash\{v_3, v_4\}$. Suppose that $Q'$ has endvertices $v_3''  \in F_3$ and $v_3' \in B$. By \ref{AppenrulepathsP2}, we can extend $Q'$ to a $(v_3, v)$-path $Q$ of length at most $m^4+2+2m^4\le4m^4$ avoiding $V(P_2)\cup (V(P_1)\backslash\{v_3\})$. Since $|B|\ge D$, we can find a $(D, m^4)$-expansion $F_3'$ of $v$ in $H'[A, B]$. Let $P_1 ' = P_1 \cup Q$. As $\ell(Q) \le 4m^4$, $P_1$ is a $(v, v_3)$-path of length at least $\ell(P_1) + 1$ and at most $\ell(P_1) + 4m^4$. Then, $(P_1',P_2, v, v_4, F_3',F_4)$ satisfies \ref{AppenrulepathsP1}-\ref{AppenrulepathsP2} with $\ell(P_1) > \ell(P)$, a contradiction. 

Recall that $|W \cup V (P_1\cup P_2)| \le 3D/m^3$. We have $\varepsilon(D)\cdot D/4 \ge|W \cup V (P_1\cup P_2)|$. Then by Lemma \ref{diamter}, there is a path $R$ of length at most $m^4$, from some $r_1 \in V (F_3)$ to some $r_2 \in V (F_4)$ avoiding $W \cup V (P_1)\cup V (P_2)\backslash\{v_3, v_4\}$. Let $Q_i$ be a path from $v_{i+2}$ to $r_i$ in $F_{i+2}$ of length at most $2m^4$ for $i\in [2]$. Then, $P_1\cup P_2 \cup Q_1\cup Q_2 \cup R$ is a $(v_1, v_2)$-path in $G - W$ of length at least $\ell(P_1)+\ell(P_2) \ge\ell$ and at most $\ell+4m^4 + 2m^4 + 2m^4 + m^4 = \ell + 9m^4$ by \ref{AppenrulepathsP2}.
\end{proof}
Next, we prove Lemma \ref{path1} by Lemmas \ref{path11} and \ref{diamter}.
\begin{proof}[{Proof of  Lemma \ref{path1}}]   Recall that $|U_1 \cup U_2| \ge 2D$ and $|V (F_3) \cup V (F_4)| = 2D$. We have $\varepsilon(2D) \cdot 2D/4 \ge  \varepsilon(n) \cdot2D/4=\varepsilon_1\cdot2D/(4m^2) \ge 2D/m^3 \ge |W|$. Then by Lemma \ref{diamter}, there is a shortest path $P_0 \subseteq G-W$ from $U_1\cup U_2$ to $V (F_3)\cup V (F_4)$ of length at most $m^4$. Without loss of generality, we can assume that $P_0$ goes from some $v_1 \in U_1$ to $\overline{v}_3\in V (F_3)$. As $F_3$ is a $(D, 2m^4)$-expansion of $v_3$, we can extend $P_0$ to a $(v_1, v_3)$-path $P$ of length at most $3m^4$.

Let $W_0 = W \cup V (P)$. Then $|W_0 | \le D/m^3 + 3m^4 + 1 \le 2D/m^3$. By Lemma \ref{path11} with $(F_1, F_2, D, m, W, \ell) = (U_2, F_4, D, m, W_0, \ell)$, there is a path $Q$ in $G-W_0$ from some $v_2 \in U_2$ to $v_4$ of length between $\ell$ and $\ell + 5m^4$. As $\ell\le \ell(P) +\ell(Q) \le \ell + 8m^4$, the paths $P$ and $Q$ are desired.
\end{proof}

\section{Proof of Lemma \ref{unit}}\label{unit-le}
In this section, we greedily find a collection of units whose interiors are pairwise disjoint. To achieve this, we first prove that every $K_{s, t}$-free graph is also dense even removing a “not big" vertex set, see Lemma \ref{average-degree2}.  Note that $d \ge \log^{200} n \ge m^{200}$. Since $n/\eta^2 \ge K$, we have $m^4 \ge  \log^{4} K$. Let $c_1 =1/(1000t)$. For sufficiently large $K$, we have
\begin{equation}\label{8m5}
c_1\eta> c_1d^{1/2} \ge 10m^{40}.
\end{equation}
and
\begin{equation}\label{8m6}
n/\eta^2\ge m^a.
\end{equation}
for any given constant $a$. 

\begin{lemma}\label{average-degree2}
Let $t\ge s\ge2$ be integers, and $x > 0$. There exists $K$ such that the following holds for each $n$ and $d$ satisfying $n \ge K \eta^2$ and $d \ge \log^{200} n$. If $G$ is an $n$-vertex $K_{s, t}$-free graph with $\delta(G)\ge d $, then for any vertex set $W \subseteq V(G)$ of size at most $\eta^2m^x$, we have $d(G - W) \ge d/2$.
\end{lemma}

\begin{proof}
Let $H = (V(G)\backslash W, W, E)$ be a bipartite subgraph of $G$, and $E$ be the set of all edges between $V(G)\backslash W$ and $W$ in $G$. Applying Lemma \ref{1954kova} with $(G, A, B, s, t) =(H, V(G)\backslash W, W, s, t)$, we have
\[
|V(G)\backslash W|\tbinom{\overline{d}(V(G)\backslash W)}{s}\le t\tbinom{|W|}{s}.
\]
That is,
\[
\frac{\overline{d}(V(G)\backslash W)\cdot(\overline{d}(V(G)\backslash W)-1) \cdots (\overline{d}(V(G)\backslash W)-s+1)}{s!}\le\frac{t}{|V(G)\backslash W|} \cdot\frac{|W| \cdot(|W|-1) \cdots (|W|-s+1)}{s!}.
\]
Note that
\[
(\overline{d}(V(G)\backslash W)-s+1)^s\le\overline{d}(V(G)\backslash W)\cdot(\overline{d}(V(G)\backslash W)-1) \cdots (\overline{d}(V(G)\backslash W)-s+1),
\]
and
\[
|W| (|W|-1) \cdots (|W|-s+1)\le |W|^s.
\]
Thus, we have
\[
(\overline{d}(V(G)\backslash W)-s+1)^s\le\frac{t}{|V(G)\backslash W|} |W|^s,
\]
and
\[
\overline{d}(V(G)\backslash W)\le\left(\frac{t}{|V(G)\backslash W|}\right)^{1/s}|W|+s-1\le t \eta^2m^x/n^{1/s}\le d/2,
\]
where the third inequality follows from (\ref{8m6}) and the choice of $K$. Thus we have $d(G-W) \ge d/2$.

\end{proof}

We give the following claim, which helps us find many vertex-disjoint hubs.
\begin{claim}\label{claim-hub}
Let $t\ge s\ge2$  be integers. There exists $K$ such that the following holds for each $n$ and $d$ satisfying $n \ge K \eta^2$, $d \ge \log^{200} n$ and any $h_1, h_2\le\eta/(100t)$. If $G$ is an $n$-vertex $K_{s, t}$-free graph with $\delta(G)\ge d $, then for any vertex set $W' \subseteq V(G)$ of size at most $\eta^2m^{52}$, we have $ G - W' $ contains an $(h_1,h_2)$-hub.
\end{claim}
\begin{proof}
By Lemma \ref{average-degree2}, $d(G-W')\ge d/2$, and then we can find a subgraph $H$ in $G-W'$ with $\delta(H)\ge d/4$. Arbitrarily choose a vertex $v$ in $H$ and let $A = N_H(v)$, so that $|A| \ge d/4$. We simply say a star is nice if it has $\eta/(100t)$ leaves and the center of this star lies in $A$.  It suffices to find $\eta/(100t)$ vertex-disjoint nice stars in $H-\{v\}$, which together with $v$ forming a $(\eta/(100t), \eta/(100t))$-hub. Let $A'\subset A $ be a maximal subset such that we can find $|A'|$ vertex-disjoint nice stars in $H - \{v\}$. Let $B$ be the union of the leaves of
all these nice stars. If $|A'|\ge \eta/(100t)$, then we are done. Otherwise, $|A \backslash A'| \ge d/4 - \eta/(100t)\ge d/8$. Each vertex in $A \backslash A'$ has fewer than $\eta/(100t)$ neighbors in $V (H) \backslash (B \cup \{v\})$, otherwise that vertex could be added to $A'$, a contradiction to the maximality of $A'$. Therefore, as $\delta(H) \ge  d/4$, each vertex in $A\backslash A'$ has at least $d/4 - \eta/(100t) -1 \ge d/8$ neighbors in $B$. Note that $H -\{v\}$ does not contain a copy of $K_{s-1, t}$ with $t$ vertices in $A \backslash A'$ and $s - 1$ vertices in $B$, since otherwise such a copy together with $v$ forms a copy of $K_{s, t}$ in $H$. Therefore, by Corollary \ref{coro2.6}, we have \[
|B| \ge \frac{d|A\backslash A'|^{1/(s-1)}}{8et} \ge \frac{\eta^2}{8^{s/(s-1)}et} \ge \frac{\eta^2}{10^4t^2}.
\] As $|B| = |A'|\eta/(100t)$, we thus have $|A'| \ge \eta/(100t)$, a contradiction.
\end{proof}

Now, we prove Lemma \ref{unit}.
\begin{proof}[{Proof of Lemma \ref{unit}}]
We choose $K$ to be sufficiently large.
By  Claim \ref{claim-hub}, we can greedily find vertex-disjoint hubs $H(w_1), \ldots, H(w_{m^{32}})$ and $H(u_1), \ldots, H(u_{\eta m^{32}})$ in $G-W$ such that each $H(w_i)$, $1 \le i \le m^{32}$, is a $(2c_1\eta, 2c_1\eta)$-hub and each $H(u_j)$, $1\le j \le \eta m^{32}$, is a $(2m^{20}, 2c_1\eta)$-hub. Indeed, let 
\[
W'=\bigcup_{i=1}^{m^{32}}V(H(u_i))\cup \left(\bigcup_{j=1}^{\eta m^{32}}V(H(u_j))\right)\cup W
\]
Then  $|W'|\le4c_1\eta^2m^{20} + 2\cdot4c_1^2\eta^2m^{32} + 2\cdot2c_1\eta\cdot2m^{20}\cdot\eta   m^{32} \le 10c_1\eta^2m^{52}$.

Recall that for a hub with a center $v$, $S_1(v)$ is the vertex set of the neighbors of $v$ in the hub, and $B_1(v) = \{v\} \cup S_1(v)$. For simplicity, let $Z=\bigcup_{i=1}^{m^{32}}B_1(w_i)\cup \left(\bigcup_{j=1}^{\eta m^{32}}B_1(u_j)\right)$. Based on the hubs we found as above. We shall  construct a unit using some vertex $w_i$ as the core vertex of the unit. Let $\mathcal{P}$ be a maximum collection of internally vertex-disjoint paths in $G-W$ under the following rules.
\stepcounter{propcounter}
\begin{enumerate}[label = ({\bfseries \Alph{propcounter}\arabic{enumi}})]
\rm\item\label{App-paths1} Each path $P_{ij}$ is a unique $w_i,u_j$-path of length at most $2m^4$.
\rm\item\label{App-paths2} Each path $P_{ij}$ does not contain any vertex in $Z\backslash (B_1(w_i)\cup B_1(u_j))$. 
\end{enumerate}
\begin{claim}\label{unit2}
There exists a vertex $w_i$ connected to at least $c_1\eta$ vertices $u_j$ via the paths in $\mathcal{P}$.
\end{claim}
\begin{proof}
Suppose to the contrary that each vertex $w_i$ is connected to fewer than $c_1\eta$ vertices $u_j$ by
the paths in $\mathcal{P}$. Let $\mathsf{Int}(\mathcal{P})$ be the set of interior vertices in all the paths in
$\mathcal{P}$. Then by \ref{App-paths1},
\[
|\mathsf{Int}(\mathcal{P})| \le 2m^4 \cdot m^{32} \cdot c_1\eta = 2c_1\eta m^{36}.
\]
Let $W_1$ be the vertex set consisting of $\mathsf{Int}(\mathcal{P})$ and $\cup_jB_1(u_j)$ if $u_j$ has been connected to at least one of the vertices $w_i$. As there are at most $m^{32}\cdot c\eta$ such vertices $u_j$, we have
\[
|W_1| \le |\mathsf{Int}(\mathcal{P})| + m^{32} c_1\eta \cdot (2m^{20} + 1)\le 2c_1\eta m^{36} + 4c_1\eta m^{52}\le c_1^2\eta^2m^{12},
\]
where the last inequality by (\ref{8m5}). For each $1 \le i \le m^{32}$, let $T_i = B_1(w_i)\backslash W_1$. Then by \ref{App-paths2} and the assumption, we have $|T_i| \ge c_1\eta$.

As the graphs $H(w_i)$ are vertex disjoint $(2c_1\eta, 2c_1\eta)$-hubs, we have $|\bigcup_{i=1}^{m^{32}}N_{H(w_i)}(T_i)| \ge2c_1\eta \cdot c_1\eta \cdot m^{32}$, and hence we have
\[
|\cup_{i=1}^{m^{32}}N_{H(w_i)}(T_i)\backslash W_1| \ge2c_1^2\eta^2m^{32}-c_1^2\eta^2m^{12}\ge c_1^2\eta^2m^{32}.
\]
Notice that there are at least $\eta m^{32}-\eta\cdot c_1m^{32}  \ge\eta m^{32}/2$ vertices $u_j$ which is not connected by a path in $\mathcal{P}$. Without loss of generality, we write these vertices $u_1$, \ldots, $u_p$, where $p \ge \eta m^{32}/2$. By \ref{App-paths2}, $\cup_{j=1}^{p}B_1(u_j)\cap W_1=\emptyset$, and we have 
\[
|\cup_{j=1}^{p}H(u_j) - W_1| \ge 2m^{20}\cdot2c_1\eta \cdot \eta m^{32}/2- c_1^2\eta^2m^{12} \ge c_1^2\eta^2m^{52}.
\]
We will apply Lemma \ref{diamter} to connect $\cup_{i=1}^{m^{32}}N_{H(w_i)}(T_i)\backslash W_1$ and $\cup_{j=1}^{p}V (H(u_j ))\backslash W_1$, while
avoiding the vertices in $\bigcup_{i=1}^{m^{32}} T_i\cup W\cup W_1$. Since $d \ge \log^{200}n\ge m^{200}$ and $n/\eta^2\ge K$ is sufficiently large, we have
\begin{equation}\label{m4}
|\cup_{i=1}^{m^{32}}T_i|+ |W| + |W_1| \le (2c_1\eta + 1)m^{32} + 2c_1\eta^2m^{20} + c_1^2\eta^2m^{12} \le 4c_1\eta^2m^{20}.
\end{equation}
Recall that $\varepsilon(x)$ is decreasing and $n/\eta^2 \ge K$ is sufficiently large, we have
\begin{equation}\label{mm2}
\varepsilon(n)=\frac{\varepsilon_1}{\log^2 (15n/\varepsilon^2\eta_2)}\ge\frac{4}{\log^3(15n/\varepsilon_2\eta^2)}\ge \frac{100}{m^{4}},
\end{equation} where $\varepsilon$ is defined in Definition \ref{subk}.
Hence, 
\[
\varepsilon(c_1^2\eta^2m^{32}) \cdot c_1^2\eta^2m^{32}/4 \ge\varepsilon(n) \cdot c_1^2\eta^2m^{32}/4\ge100\cdot c_1^2\eta^2m^{32}/(4m^{4})\ge 4c_1^2\eta^2m^{28}\geq 4c_1\eta^2m^{20},
\]
where the second inequality by (\ref{mm2}).

Thus, by (\ref{m4}) and Lemma \ref{diamter}, there is a shortest path $\overline{P}_{k_1k_2}$ of length at most $2\log^3(15n/\varepsilon_2\eta^2)/\varepsilon_1+1\le  m^4$
from $T_{k_1}$ to $V (H(u_{k_2}))$ avoiding $W$ and $W_1$ for some $k_1\in [m^{32}],k_2\in [p]$, and $\overline{P}_{k_1k_2}$ can be extended to a $(w_{k_1}, u_{k_2})$-path $P_{k_1k_2}$ of length at most $1 + m^4 + 2 \le 2m^4$ in $G - (W \cup W_1)$, a contradiction to the maximality of $\mathcal{P}$, as claimed.
\end{proof}

Hence, we may assume that $w_i$ connects to $c_1\eta$ vertices $u_1,u_2,\ldots, u_{c_1\eta}$ by Claim \ref{unit2}. For all $(2m^{20}, 2c_1\eta)$-hubs corresponding to each $u_j$, $j\in [c_1\eta]$, if we can find an $(m^{20}, c_1\eta)$-hub which is disjoint from $\cup_{j=1}^{c_1\eta}\mathsf{Int}(P_{ij})$, then all such $(m^{20}, c_1\eta)$-hubs together with $\cup_{j=1}^{c_1\eta}P_{ij}$ form a $(c_1\eta, m^{20}, c_1\eta, 2m^4)$-unit, as desired. So it remains to find a $(m^{20}, c_1\eta)$-hub centered at $u_j$ as required above for each $j \in [c_1\eta]$. By \ref{App-paths1}, we have that each path $P_{ij}$ has length at most $2m^4$ and then $|\cup_{j=1}^{c_1\eta}V(P_{ij})| \le 2c_1\eta m^4$. It is possible to  take a set of $m^{20}$ vertices $v\in S_1(u_j)$ for each $j\in [c_1\eta]$ along with $c_1\eta$
vertices in $S_1(v)$ avoiding $\cup_{j=1}^{c_1\eta}V(P_{ij})$, forming the desired $(m^{20}, c_1\eta)$-hub.
\end{proof}

\section{Proof of Lemma \ref{adjuster2}}\label{3.12-le}
In this section, we shall prove Lemma \ref{adjuster2}. We first claim that there are varieties of simple adjusters in $G$ (see Lemma \ref{largee}), which can be linked to many \emph{octopuses} (see Definition \ref{oc}). Now,  we give the following lemma, which helps us find many $(D, m^4/2, 1)$-adjusters, and we postpone the proof of this lemma later.

\begin{lemma}\label{largee}
Suppose $1/n, 1/d\ll 1/K\ll \varepsilon_1, \varepsilon_2, 1/s, 1/t$, and $t, s\in \mathbb{N}$ satisfy $t\ge s\ge2$,  and $d\ge \log^{200}n$. If $G$ is an $n$-vertex $K_{s,t}$-free $(\varepsilon_1, \varepsilon_2\eta^2)$-expander with $\delta(G) \ge d$ and $n\ge K\eta^2$, then for any vertex set $W\subseteq V(G)$ satisfying $|W| \le 5D$ with $D=\frac{\eta^2m^{20}}{10^{10}t}$, $G-W$ contains a $(D, m^4/2, 1)$-adjuster.
\end{lemma}

\begin{proof}[{Proof of Lemma \ref{adjuster2}}]
We prove the lemma by induction on $r$. When $r = 1$, by Lemma \ref{largee}, there is  a $(D, m^4/2, 1)$-adjuster in $G - W$. Next, we assume that there  exists a $(D, m^4, r)$-adjuster in $G-W$ for  some $r$ with $1 < r \le \frac{1}{20}\eta^ 2m^{12}$, denoted by $\mathcal{A}_1 = (v_1, F_1, v_2, F_2, A_1)$. Let $W_1 = W \cup A_1\cup V (F_1)\cup V (F_2)$. Then $|W_1| \le 4D$. By Lemma \ref{largee}, $G - W_1$ contains a $(D, m^4/2, 1)$-adjuster, denoted by  $\mathcal{A}_2= (v_3, F_3, v_4, F_4, A_2)$. Note that $|F_1\cup F_2| = |F_3\cup F_4|= 2D $, and $|W \cup A_1\cup A_2| \le \frac{D}{\log^3n/\eta^2}  + 20rm^4  \le \frac{2D}{\log^3n/\eta^2}$. Since $\varepsilon(2D)$ is decreasing when $2D\ge k$ and  $\varepsilon(n,\varepsilon_1,k)=\frac{\varepsilon_1}{\log^2n/\eta^2}\ge \frac{100}{\log^3n/\eta^2}$ (here $\varepsilon(n,\varepsilon_1,k)$ is defined in Definition \ref{subk}), we have $\varepsilon(2D)\cdot2D/4 \ge \varepsilon(n)\cdot2D/4 \ge \frac{2D}{\log^3n/\eta^2} \ge |W \cup A_1 \cup A_2|$. By Lemma \ref{diamter}, there exists a path $P$ of length at most $m^4$, from $V (F_1)\cup V (F_2)$ to $V (F_3)\cup V (F_4)$ avoiding $W \cup A_1\cup A_2$. We may assume that $P$ is a path from $V (F_1)$ to $V (F_3)$.

Due to the fact that $F_1$ and $F_3$ are both $(D, m^4)$-expansions of $v_1$ and $v_3$, respectively, $P$ can be extended to a $(v_1, v_3)$-path $Q$ of length at most $3m^4$ via these two expansions. We aim to claim that $(v_2, F_2, v_4, F_4, A_1 \cup A_2 \cup V (Q))$ forms a $(D, m^4, r+ 1)$-adjuster $\mathcal{A}$. It is easy to see that $\mathcal{A}$ satisfies \ref{Ad1} and \ref{Ad2} of Definition \ref{defad}. As $|A_1\cup A_2\cup V (Q)| \le 10m^4r+10m^4/2+3m^4 \le 10(r+1)m^4$, \ref{Ad3} holds. For \ref{Ad4}, let $\ell = \ell(A_1)+\ell(A_2)+\ell(Q)$. For every $i \in \{0, 1, \ldots, r + 1\}$, there are some $i_1 \in \{0, 1, \ldots, r\}$ and $i_2 \in \{0, 1\}$ such that $i = i_1 + i_2$. Let $P_1$ be a $(v_1, v_2)$-path in $G[A_1 \cup \{v_1, v_2\}]$ of length $\ell(A_1) + 2i_1$ and $P_2$ be a $(v_3, v_4)$-path in $G[A_2 \cup \{v_3, v_4\}]$ of length $\ell(A_2) + 2i_2$. Then $P_1 \cup Q \cup P_2$ is a $(v_2, v_4)$-path in $G[A_1 \cup A_2 \cup V (Q)]$ of length $\ell + 2i$, as claimed.
\end{proof}

In order to prove Lemma \ref {largee}, we introduce a structure from \cite{2023yang}, called \emph{Octopus}, which is comprised of many vertex-disjoint simple adjusters.

\begin{definition}
\rm
(Octopus).\label{oc} Given integers $r_1, r_2, r_3, r_4 > 0$, an $(r_1, r_2, r_3, r_4)$-octopus $\mathcal{B}= (A, R, \mathcal{D}, \mathcal{P})$ is a graph consisting of a \emph{core} $(r_1, r_2, 1)$-adjuster $A$, one of the ends of $A$, called $R$, and
\stepcounter{propcounter}
\begin{enumerate}[label = ({\bfseries \Alph{propcounter}\arabic{enumi}})]
\rm\item\label{App-oco1} a family $\mathcal{D}$ of $r_3$ vertex-disjoint $(r_1, r_2, 1)$-adjusters, which are disjoint from $A$, and

\rm\item\label{App-oco2} a minimal family $\mathcal{P}$ of internally vertex-disjoint paths of length at most $r_4$, such that each adjuster in $\mathcal{D}$ has at least one end which is connected to $R$ by a subpath from a path in $\mathcal{P}$, and all the paths are disjoint from all center sets of the adjusters in $\mathcal{D} \cup \{A\}$. Obviously,  $|\mathcal{P}| \le |\mathcal{D}|$.
\end{enumerate}
\end{definition}
\begin{proof}[{Proof of Lemma \ref{largee}}]
First, we claim that there are $m^{240}$ pairwise disjoint $(\frac{\eta^2}{6000t}, \frac{m^4}{600}, 1)$-adjusters in $G - W$. Let $W'$ be the vertex set of these adjusters. Then $|W'|\le(2\cdot\frac{\eta^2}{6000t} + 10\cdot \frac{m^4}{600})\cdot m^{240} \le \frac{\eta^2m^{240}}{30t}$. Applying Lemma \ref{average-degree2} with $W'=W$, we have $d(G - W' ) \ge d/2$, and by Corollary \ref{coro} with $G = G-W'$, there exists a bipartite $(\varepsilon_1, \varepsilon_2\eta)$-expander $H \subseteq G-W'$ with $\delta(H) \ge \frac{d}{16}$. Thus, there exists a shortest cycle $C$ in $G'$ of length at most $\frac{m^4}{16}$ and we denoted by $2r$ the length of $C$. Arbitrarily pick two vertices $v_1$, $v_2 \in V (C)$ of distance $r-1$ on $C$. Since $\delta(H)\geq \frac{d}{16}$, we can find a subset $A_i\subset N_{H-C}(v_i)$ of size $\frac{d}{400}$ for each $i\in [2]$ such that $A_1\cap A_2=\emptyset$. Let $B_1=N_{H-C}(A_1)\backslash  A_2$ and $B_2=N_{H-C}(A_2)\backslash  A_1$. Note that $G[A_i, B_i]$ does not contain a copy of $K_{s-1, t}$ with $t$ vertices in $A_i$ and $s - 1$ vertices in $B_i$. For each $v\in A_i$, there are at least $\frac{1}{16}d-|C|-|A_1|-|A_2|\ge \frac{1}{32}d$ neighbors in $B_i$. By Corollary \ref{coro2.6} with $\delta: =\frac{d}{32}$ yields that 
\[
|B_i|\ge\frac{1}{et}\left(\frac{d}{32}\right)\left(\frac{d}{400}\right)^{\frac{1}{s-1}}\ge \frac{\eta^2}{3000t}.
\]

Therefore we can find two vertex-disjoint $(\frac{\eta^2}{6000t}, 2)$-expansion of $v_1$ and $v_2$ respectively, say $F_1$ and $F_2$, such that $|N(v_i')\cap V(F_i)|\le d$ for each $v_i'\in N_{F_i}(v_i)$ for each $i\in[2]$. Hence,  we get a $(\frac{\eta^2}{6000t}, \frac{m^4}{600}, 1)$-adjuster as desired by combining $F_1$, $F_2$  and $C$.  

An adjuster is \emph{touched} by a path if they intersect in at least one vertex, and \emph{untouched} otherwise. The following claim helps us to find  internally vertex-disjoint short paths by connecting a vertex set to many ends of different adjusters.

\begin{claim}\label{adjuster3}Let $x \ge10$ and $X \subseteq V (G)$ be an arbitrary vertex set of size at most $\eta^2m^x/2$. Let $B \subseteq G \backslash X$ be a graph with order at least $\frac{\eta^2m^{x+8}}{3000t}$ and $\mathcal{U}$ be a subfamily of $(\frac{\eta^2}{6000t}, \frac{m^4}{600}, 1)$-adjusters in $G \backslash (X \cup V (B))$ with $|\mathcal{U}| \ge m^{3x}$. Let $\mathcal{P}_B$ be a maximum collection of internally vertex-disjoint paths of length at most $m^4/6$ in $G- X$, each connecting $V (B)$ to one end from different adjusters in $\mathcal{U}$. Then $V (B)$ can be connected to $1200m^{x+20}$ ends from different adjusters in $\mathcal{U}$ via a subpath from a path in $\mathcal{P}_B$.
\end{claim}
\begin{proof}
Suppose to the contrary that $V (B)$ is connected to less than $1200m^{x+20}$ ends from different adjusters in $\mathcal{U}$ via a subpath from a path in $\mathcal{P}_B$. Then we have $|\mathsf{Int}(\mathcal{P}_B)| \le1200m^{x+20}\cdot m^4/6 = 200m^{x+24}$, and there are at least $m^{3x} -1200m^{x+20}$ adjusters in $\mathcal{U}$ untouched by the paths in $\mathcal{P}_B$. We arbitrarily pick $m^{x+5}$ adjusters among those untouched adjusters, and let $B_0$ be the union of their ends. Then $|B_0 | =  \frac{2\eta^2 m^{x+5}}{6000t} = \frac{\eta^2m^{x+5}}{3000t}$. As $d\ge \log^{200} n \ge m^{200}$, we have $|X \cup \mathsf{Int}(\mathcal{P}_B)| \le \eta^2m^x/2+ 200m^{x+24} \le \eta^2m^x$.  According to  (\ref{mm2}), we have $\varepsilon(\frac{\eta^2m^{x+5}}{3000t})\cdot \frac{\eta^2m^{x+5}}{4\cdot3000t}\ge\varepsilon(n)\cdot \frac{\eta^2m^{x+5}}{4\cdot3000t} \ge \eta^2m^{x} \ge |X \cup P|$. Then by Lemma \ref{diamter}, there is a path of length at most $m^4/6$ from $V (B)$ to $V (B_0 )$, avoiding $X \cup \mathsf{Int}(\mathcal{P}_B)$, a contradiction to the maximality of $\mathcal{P}_B$.
\end{proof}

Next, we will use $m^{240}$ pairwise disjoint $(\frac{\eta^2}{6000t}, \frac{m^4}{600}, 1)$-adjusters to construct $m^{40}$ octopuses, and further to build a $(D, m^4/2, 1)$-adjuster. Let $Z$ be the union of the center sets and core vertices of all those adjusters.

\begin{claim}\label{adjuster4}
There are $m^{40}$ $(\frac{\eta^2}{6000t}, m^4/6, 600m^{20}, m^4/6)$-octopuses $\mathcal{B}_j = (A_j , R_j , \mathcal{D}_j, \mathcal{P}_j )$ $(1 \le  j\le m^{40})$ in $G  - W$ such that the following rules hold.
\stepcounter{propcounter}
\begin{enumerate}[label = ({\bfseries \Alph{propcounter}\arabic{enumi}})]
\rm\item\label{App-ocotop-build1}  $A_j$ are pairwise disjoint adjusters and $A_i \notin \mathcal{D}_j$, $1 \le i\neq j \le m^{40}$.

\rm\item\label{App-ocotop-build2} $\mathcal{D}_j$ contains every adjuster which intersects at least one path in $\mathcal{P}_j$ , $1 \le j \le m^{40}$.

\rm\item\label{App-ocotop-build3} Paths in $\mathcal{P}_i$ are vertex disjoint from $Z$ and $A_j$, $1 \le i\neq j \le m^{40}$.

\rm\item\label{App-ocotop-build4} All paths in $\mathcal{P}$ are pairwise vertex disjoint.
\end{enumerate}
\end{claim}
\begin{proof}
Suppose that there exist $q<m^{40}$ octopuses satisfying the above rules so far.  Let $U$ be the union of the vertex sets of the ends of the core adjusters of the existed octopuses. Then we have $|U| \le \frac{m^{40}\cdot 2\eta^2}{6000t} = \frac{\eta^2m^{40}}{3000t}$. We say an adjuster \emph{used} if it is used to construct an octopus, and \emph{unused} otherwise. There are at most $m^{40}\cdot (600m^{20} + 1)$ adjusters are used until now, so there are more than $m^{160}$ unused adjusters. Let $P = \bigcup_{j=1}^{q} V (\mathcal{P}_j )$. Then $|P| \le m^4/6 \cdot 600m^{20}\cdot m^{40} \le m^{65}$. Arbitrarily pick a subfamily $\mathcal{B}$ of $m^{48}$ unused adjusters, and let $B$ be the union of their ends. Then $|B| = \frac{m^{48}\cdot2\eta^2}{6000t} = \frac{\eta^2m^{48}}{3000t}$. Let $\mathcal{U}$ be the family of unused adjusters, except for the adjusters we picked above. Then we  have $|\mathcal{U}|\ge m^{120}$. Let $W_1 = W \cup Z$. Then $|W_1 \cup U \cup P| \le (5D + m^{240}\cdot \frac{m^4}{30}) + \frac{\eta^2m^{40}}{3000t} + m^{65} \le  \eta^2m^{40}/2$ as $d \ge \log^{200} n \ge m^{200}$. Then by Claim \ref{adjuster3} with $(B, \mathcal{U}, x, X) = (B, \mathcal{U}, 40, W_1 \cup U \cup P)$, $V (B)$ can be connected to $1200m^{60}$ ends from different adjusters in $\mathcal{U}$ via some internally vertex-disjoint paths of length at most $m^4/6$ in $G - (W_1 \cup U \cup P)$. By the pigeonhole principle, we can find an adjuster in $\mathcal{B}$, say $A_{q+1}$ such that $A_{q+1}$ has an end $R_{q+1}$ connected to at least $600m^{20}$ adjusters, say $\mathcal{D'}_{q+1}$, via a subfamily of internally $\mathcal{P'}_{q+1}$.

Denote by $L_{q+1}$ the other end of $A_{p+1}$. Obviously, \ref{App-ocotop-build1}, \ref{App-ocotop-build2} and \ref{App-ocotop-build3} hold. As we find paths in $\mathcal{P}_{q+1}$ avoiding $W_1 \cup U \cup P$, \ref{App-ocotop-build4} holds. Thus, $A_{q+1}$, $R_{q+1}$, $\mathcal{D}_{q+1}$ and $\mathcal{P}_{q+1}$ form a $(\frac{\eta^2}{6000t}, \frac{m^4}{600}, 600m^{20}, m^4/6)$-octopus.
\end{proof}

Now we have $m^{40}$ $(\frac{\eta^2}{6000t}, m^4/6, 600m^{20}, m^4/6)$-octopuses $\mathcal{B}_j = (A_j, R_j, \mathcal{D}_j, \mathcal{P}_j )$, $1 \le j \le m^{40}$. Let $B_1=\bigcup_{j=1}^{m^{40}}L_j$. Then $|B_1| = \frac{m^{40}\cdot \eta^2}{6000t}$. It is easy to deduce that there are at most $m^{40}\cdot (600m^{20} + 1)$ adjusters used, and thus at least $m^{160}$ adjusters unused. Let $\mathcal{U}_0$ be the family of these unused adjusters. By Claim \ref {adjuster3} with $(B, \mathcal{U}, x, X) = (B_1, \mathcal{U}_0 , 32, W_1\cup P \cup Q)$, we shall show that $V (B_1)$ can be connected to $1200m^{52}$ ends from distinct adjusters in $U_0$ via some internally vertex-disjoint paths of length at most $m^4/6$ in $G - W_1 - P - Q$. Reset $P =\bigcup^{m^{40}}_{i=1}V(\mathcal{P}_j)$, then $|P| \le m^{65}$. By definition, inside each $\mathcal{B}_j = (A_j, R_j, \mathcal{D}_j, \mathcal{P}_j)$, $j \in [m^{40}]$, every adjuster $A \in \mathcal{D}_j$ intersects $V (\mathcal{P}_j)$ and thus there exists a shortest path in $A$ of length at most $\frac{m^4}{600}$ connecting a core vertex of $A$ to $V(\mathcal{P}_j)$, and denote by $\mathcal{Q}_j$ the disjoint union of such paths taken over all adjusters in $D_j$. Let $Q = \bigcup_{j=1}^{m^{40}}V (\mathcal{Q}_j)$. Then $|Q| \le m^{40}\cdot 600m^{20}\cdot (m^4/600+1) \le m^{65}$. Note that $|W_1 \cup P \cup Q| \le (5D +m^{240}\cdot \frac{m^4}{30}) +m^{64} +m^{65} \le  \eta^2m^{32}/2$ as  $d \ge \log^{200} n \ge m^{200}$.

By the pigeonhole principle, there exists a core adjuster $A_k$ such that $L_k$ is connected to a family $\mathcal{D}_k'$ of at least $600m^{20}$ adjusters, via a subfamily of internally vertex-disjoint paths, denote by $\mathcal{P}_k'$. Then $A_k$, $L_k$, $ \mathcal{D}_k'$ and $\mathcal{P}_k'$ form a $(\frac{\eta^2}{6000t}, \frac{m^4}{600}, 600m^{20}, m^4/6)$-octopus. Note that $(A_k, R_k, \mathcal{D}_k, \mathcal{P}_k)$ is also a $(\frac{\eta^2}{6000t}, \frac{m^4}{600}, 600m^{20}, m^4/6)$-octopus. For the adjuster $A_k$, denote by $C_k$ the center vertex set of $A_k$, and note that $L_k$, $R_k$ are $(\frac{\eta^2}{6000t}, \frac{m^4}{600})$-expansions of vertices $v_1$, $v_2$, respectively. Let $F_1' := G[V (L_k)\cup V (\mathcal{P}_k')\cup V (\mathcal{D}_k')]$, and $F_2'$ be the component of $G[V (R_k)\cup V (\mathcal{P}_k)\cup V (\mathcal{D}_k)]-V (\mathcal{P}_k')$ containing $v_2$. Indeed, paths in $\mathcal{P}_k$ and $\mathcal{P}_k'$ are disjoint from $Z$, and $V (P_k)$ and $V (\mathcal{P}_k')$ are disjoint. Recall that for every adjuster in $\mathcal{D}_k$, every vertex in the ends of the adjuster has at most $d$ neighbors in the adjuster, except for its core vertices. As $d \ge \log^{200} n \ge m^{200}$, and $V (P_k')$ is disjoint from $Z$ and $Q$, $F_2'$ has size at least $|V (\mathcal{D}_k)| - \eta|V (\mathcal{P}_k')| \ge  2\cdot600m^{20}\cdot \frac{\eta^2}{6000t}-\eta \cdot  600m^{20}\cdot m^4/6 \ge \eta^2m^{20}$, and the distance between $v_2$ and each $v \in V (F_2')$ is at most $\frac{m^4}{600} + \frac{m^4}{6} + \frac{m^4}{600} + \frac{m^4}{32} +\frac{m^4}{600} \le m^4/2$. Then by Proposition \ref{3.4}, there exists a subgraph of $F_2'$, denoted by $F_2$, which is a $(\eta^2m^{20}, m^4/2)$-expansion of $v_2$. Similarly, we can find $F_1$, which is a $(\eta^2m^{20}, m^4/2)$-expansion of $v_1$. Recall that $C_k \cup \{v_1, v_2\}$ is an even cycle of length $2r_0 \le \frac{m^4}{16}$, and the distance between $v_1$ and $v_2$ on $C_k \cup \{v_1, v_2\}$ is $r' - 1$. Thus, $(v_1, F_1, v_2, F_2, C_k)$ is a $(\eta^2m^{20}, m^4/2, 1)$-adjuster, and by Proposition \ref{3.4}, there exists a $(D, m^4/2, 1)$-adjuster in $G - W$.
\end{proof}

\section{Proof of Theorem \ref{liusparse}}\label{4.3-le}
In this section we give a proof of Theorem \ref{liusparse}.  To achieve this, we introduce a lemma from  Liu and Montgomery \cite{2023liu}.

\begin{lemma}[Lemma 3.11 in \cite{2023liu}]\label{liu3.11}
For each $k \in \mathbb{N}$ and any $0 < \varepsilon_1, \varepsilon_2 < 1$, there exists $d_0=d_0(\varepsilon_1, \varepsilon_2, k)$ such that the following holds for each $n \ge d \ge d_0$.

Suppose that $G$ is an $n$-vertex bipartite $(\varepsilon_1, \varepsilon_2d)$-expander with $\delta(G) \ge d-1$.
Let $m = \frac{40}{\varepsilon_1 } \log^3 n$. Let $C$ be a shortest cycle in $G$, and let $x_1, \ldots, x_k$ be distinct
vertices in $G$. For each $i$, $j \in [k]$, let $D_{i,j} \in [1, \log^{5k} n]$. Then, there are graphs $F_{i,j} \subseteq G$, $i, j \in [k]$, such that the following hold.

(i) For each $i, j \in [k]$, $F_{i,j}$ is a $(D_{i,j} , 5m)$-expansion around $x_i$ which contains no vertices other than $x_i$ in $V (C) \cup \{x_1, \ldots, x_k\}$.

(ii) The sets $V (F_{i,j} ) \backslash \{x_i\}$, $i, j \in [k]$, are pairwise disjoint.
\end{lemma}
The following Lemma can find  a short path to connect two vertex sets together while avoiding a vertex set, and we postpone its proof later.  
\begin{lemma}\label{5.3-path}
Let $ t\ge s\ge2$ be integers. There exists some $\varepsilon_1 > 0$ such that, for any $0 < \varepsilon_2 < 1/5$ and $k \ge 10$, there exists $d_0=d_0(\varepsilon_1, \varepsilon_2, k)$ such that the following holds for each $n \ge d \ge d_0$. Let $G$ be an $n$-vertex $K_{s,t}$-free bipartite $(\varepsilon_1, \varepsilon_2d)$-expander with $\delta(G) \ge d$.

Suppose $\log^{10} n \le D \le \log^k n$, and $U \subseteq V (G)$ with $|U| \le \frac{D}{2 \log^3 n}$, and let $m = \frac{800}{\varepsilon_1 } \log^3 n$. Suppose $F_1$, $F_2 \subseteq G -U $ are vertex disjoint such that $F_i$ is a $(D, m)$-expansion of $v_i$, for each $i \in [2]$. Let $\log^7 n \le \ell \le n/ \log^{12}n$ be such that $\ell
=\pi(v_1, v_2, G)$ mod 2. Then, there is a $(v_1, v_2)$-path with length $\ell$ in $G - U$.
\end{lemma}
Now, we prove Theorem \ref{liusparse} by Lemmas \ref{liu3.11} and \ref{5.3-path}.
\begin{proof}[Proof of Theorem \ref{liusparse}]
Let $\varepsilon_1 > 0$ be such that the property in Lemma \ref{5.3-path} holds.
Let $k = 10$, $d_0 = d_0(\varepsilon_1, \varepsilon_2)$ be large and $n \ge d \ge d_0$. Suppose then that $H$ is an $n$-vertex  $K_{s,t}$-free bipartite $(\varepsilon_1, \varepsilon_2d)$-expander with $\delta(H) \ge d$ and let $x$, $y \in V (H)$ be distinct. Let  $\ell \in [\log^7 n, n/ \log^{12}n]$ satisfy $\ell = \pi(x, y, H)$ mod 2. We will show
that $H$ contains an $x, y$ path with length  $\ell$.

Let $m = \frac{800}{\varepsilon_1}\log^3n$  and $D = \log^{10} n$. Then, by Lemma \ref{liu3.11} (applied with $C$ taken to be an arbitrary shortest cycle in $H$), there are vertex disjoint graphs $F_x$, $F_y \subseteq H$ so that $F_x$ is a $(D,m)$-expansion of $x$ and $F_y$ is a $(D,m)$-expansion of $y$. Then, by Lemma \ref{5.3-path} with $U = \emptyset$, there is an $(x, y)$-path with length $\ell$ in $H$, as required.
\end{proof}

\noindent
\textbf{Proof of Lemma \ref{5.3-path}.} \\

\indent
We first give a lemma to  find  a $(D, m, r)$-adjuster  while deleting a vertex set.

\begin{lemma}\label{r-adjuster}
There exists some $\varepsilon_1 > 0$ such that, for any $0 < \varepsilon_2 < 1/5$ and $k \ge 10$,
there exists $d_0 = d_0(\varepsilon_1, \varepsilon_2, k)$ such that the following holds for each $n \ge d \ge d_0$. Suppose that $G$ is an $n$-vertex $K_{s,t}$-free bipartite $(\varepsilon_1, \varepsilon_2d)$-expander with $\delta(G) \ge d$. Let $m = \frac{800}{\varepsilon_1}\log^3n$. Suppose $\log^{10} n \le D\le \log^k n$, $1 \le r \le 30m$ and $U\subseteq V (G)$ with $|U| \le D$. Then, there is a $(D, m, r)$-adjuster in $G - U$.
\end{lemma}

The following corollary is from \cite{2023liu}, which helps us to connect two sets of vertices using two vertex-disjoint paths  while avoiding a smaller vertex set.
\begin{corollary}[Corollary 3.15 in \cite{2023liu}]\label{5.3-path2}
For any $0 < \varepsilon_1, \varepsilon_2 < 1$, there exists $d_0 = d_0(\varepsilon_1, \varepsilon_2)$ such that the following holds for each $n\ge d\ge d_0$. Suppose that $G$ is an $n$-vertex bipartite $(\varepsilon_1, \varepsilon_2d)$-expander with $\delta(G) \ge d$.

Let $\log^{10} n \le D \le  n/ \log^{10} n$, $\frac{100}{\varepsilon_1} \log^3 n \le m \le \log^4n$ and $\ell \le  n/\log^{10}n$. Let $A \subseteq V (G)$ satisfy $|A| \le  D/\log^3 n$. Let $F_1, \ldots, F_4 \subseteq G - A$ be vertex-disjoint subgraphs and $v_1, \ldots, v_4$ be vertices such that, for each $i \in [4]$, $F_i$ is a $(D,m)$-expansion of $v_i$. Then, $G - A$ contains vertex-disjoint paths $P$ and $Q$ with  $\ell\le \ell(P) +  \ell(Q) \le\ell+ 22m$ such that both $P$ and $Q$ connect $\{v_1, v_2\}$ to $\{v_3, v_4\}$.
\end{corollary}
Now, we finish the proof of Lemma \ref{5.3-path}.
\begin{proof}[Proof of Lemma \ref{5.3-path}]
There is a $(D, m, 22m)$-adjuster, $\mathcal{A} = (v_3, F_3, v_4, F_4, A)$ in $G - U$, and $\ell (\mathcal{A}) \le |A| + 1 \le 230m^2$ by Lemma \ref{r-adjuster}. Let $\bar{\ell} = \ell - 22m - \ell (\mathcal{A})$. Then $0 \le \bar{\ell}\le \frac{n}{ \log^{12}n}$. As $|A\cup U| \le 230m^2+\frac{D}{2 \log^3 n} \le \frac{D}{\log^3 n}$, there are paths $P$ and $Q$ in $G-U-A$ which are vertex disjoint, both connect $\{v_1, v_2\}$ to $\{v_3, v_4\}$ by Corollary \ref{5.3-path2}. Then $\bar{\ell} \le\ell (P)+\ell (Q) \le \bar{\ell} +22m$. Without loss of generality, we assume that $P$ is a $(v_1, v_3)$-path and $Q$ is a $(v_2, v_4)$-path. Now, $0 \le \ell -\ell (P) -\ell (Q) -\ell (A) \le 22m$. As $\mathcal{A}$ is a $(D, m, 22m)$-adjuster, there is a $v_3$, $v_4$-path in $G[A \cup \{v_3, v_4\}]$ with length $\ell (\mathcal{A})$, and therefore $\ell (\mathcal{A}) =\pi(v_3, v_4, G)$ mod 2. Then, as $\ell (P) = \pi(v_1, v_3, G)$ mod 2, $\ell (Q) = \pi(v_2, v_4, G)$ mod 2, $\ell = \pi(v_1, v_2, G)$ mod 2 and $\pi(v_1, v_2, G) = \pi(v_1, v_3, G) + \pi(v_3, v_4, G) +\pi(v_4, v_2, G)$ mod 2, we have $\ell - \ell (P) - \ell (Q) - \ell (\mathcal{A}) = 0$ mod 2. That is, there is some $i \in \mathcal{N}$ with $2i = \ell - \ell (P) - \ell (Q) - \ell (\mathcal{A})$, where $i \le 11m$.
Therefore, by the property of the adjuster, there is a $(v_3, v_4)$-path, $R$ say, with length $ \ell (\mathcal{A})+2i = \ell - \ell (P) - \ell (Q)$ in $G[A\cup\{v_3, v_4\}]$. Thus, $P \cup R\cup Q$ is a $(v_1, v_2)$-path with length $\ell$ in $G - U$.
\end{proof}

\noindent
\textbf{Proof of Lemma \ref{r-adjuster}}\\

\indent
The following  lemma shows that a simple adjuster can be found even removing a vertex set.
\begin{lemma}\label{1-adjuster}
There exists some $\varepsilon_1 > 0$ such that, for any $0 < \varepsilon_2 < 1$ and $k \in \mathbb{N}$,
there exists $d_0 = d_0(\varepsilon_1, \varepsilon_2, k)$ such that the following holds for each $n \ge d \ge d_0$. Suppose that $G$ is an $n$-vertex $K_{s,t}$-free bipartite $(\varepsilon_1, \varepsilon_2d)$-expander with $\delta(G) \ge d$. Let $m = \frac{200}{\varepsilon_1}\log^3n$. Suppose $D\le \log^k n$, $1 \le r \le 30m$ and $U\subseteq V (G)$ with $|U| \le 10D$. Then, there is a $(D, m, 1)$-adjuster in $G - U$.
\end{lemma}
The following Lemma provides a short path to connect two vertex sets together while avoiding a vertex set.  
\begin{lemma}[Lemma 3.4 in \cite{2023liu}]\label{liu-diam}
For each $0 < \varepsilon_1, \varepsilon_2 < 1$, there exists $d_0 = d_0(\varepsilon_1, \varepsilon_2)$ such that the following holds for each $n \ge d \ge d_0$ and $x \ge 1$. Suppose $G$ is an $n$-vertex $(\varepsilon_1, \varepsilon_2d)$-expander with $\delta(G) \ge d-1$. Let $A, B \subseteq V (G)$ with $|A|$, $|B| \ge x$, and let $W \subseteq V (G) \backslash (A \cup B)$ satisfy $|W| \log^3n \le 10x$. Then, there is a path from $A$ to $B$ in $G - W$ with length at most $\frac{40}{\varepsilon_1} \log^3 n$.
\end{lemma}
\begin{proof}[Proof of Lemma \ref{r-adjuster}]
We prove the lemma by induction on $r$. When $r = 1$, it is easy to find a $(D, m, 1)$-adjuster in $G - W$. Next, we assume that there  exists a $(D, m, r)$-adjuster in $G-W$ for  some $r$ with $1 \le r < 30m$, denoted by $\mathcal{A}_1 = (v_1, F_1, v_2, F_2, A_1)$. Let $W_1 = W \cup A_1\cup V (F_1)\cup V (F_2)$, and we have $|W_1| \le 4D\le\log^{2k}n$. By the arguments mentioned above, we have that $G - W_1$ contains a $(D,  m, 1)$-adjuster, denoted by  $\mathcal{A}_2= (v_3, F_3, v_4, F_4, A_2)$. We claim that there exists a path $P$ of length at most $m$, from $V (F_1)\cup V (F_2)$ to $V (F_3)\cup V (F_4)$ avoiding $W \cup A_1\cup A_2$. Note that $|F_1\cup F_2| = |F_3\cup F_4|= 2D $, and $|W \cup A_1\cup A_2| \le \frac{D}{\log^3n}  + 20rm  \le \frac{2D}{\log^3n}$. Then by Lemma \ref{liu-diam}, such path exists, as claimed.

We may assume that $P$ is a path from $V (F_1)$ to $V (F_3)$. Then we can get a $(v_1, v_3)$-path $Q\subseteq (F_1 \cup P \cup F_3)$ of length at most $3m$, due to $F_1$ and $F_3$ are $(D, m)$-expansions of $v_1$ and $v_3$, respectively. We claim that $(v_2, F_2, v_4, F_4, A_1 \cup A_2 \cup V (Q))$ is a $(D, m, r+ 1)$-adjuster. Indeed, $V(F_2)\cap V(F_4)\cap (A_1 \cup A_2 \cup V (Q))=\emptyset$ and $|A_1\cup A_2\cup V (Q)| \le 10mr+10\cdot(m/2)+3m \le 10(r+1)m$. Let $\ell = \ell(A_1)+\ell(A_2)+\ell(Q)$. For every $i \in \{0, 1, \ldots, r + 1\}$, there are some $i_1 \in \{0, 1, \ldots, r\}$ and $i_2 \in \{0, 1\}$ such that $i = i_1 + i_2$. Let $P_1$ be a $(v_1, v_2)$-path in $G[A_1 \cup \{v_1, v_2\}]$ of length $\ell(A_1) + 2i_1$ and let $P_2$ be a $(v_3, v_4)$-path in $G[A_2 \cup \{v_3, v_4\}]$ of length $\ell(A_2) + 2i_2$. Hence, $P_1 \cup Q \cup P_2$ is a $(v_2, v_4)$-path in $G[A_1 \cup A_2 \cup V (Q)]$ of length $\ell + 2i$.
\end{proof}

\noindent
\textbf{Proof of Lemma \ref{1-adjuster}}
Next, we give a proof of Lemma \ref{1-adjuster}. A vertex set $A$ has \emph{k-limited} contact with a vertex set $X$ in a graph $H$ if, for each $i \in \mathbb{N}$,
\[
|N_H(B_{H-X}^{i-1}(A)) \cap X| \le ki.
\]
We need  some lemmas from Liu and Montgomery \cite{2023liu}.
\begin{lemma}[Lemma 3.7 in \cite{2023liu}]\label{liu3.7}
For each $0 < \varepsilon_1 < 1$, $0 < \varepsilon_2 < 1/5$ and $k \in \mathbb{N}$, there exists $d_0=d_0(\varepsilon_1, \varepsilon_2, k)$ such that the following holds for each $n \ge d \ge d_0$. Suppose that $G$  is an $n$-vertex bipartite $(\varepsilon_1, \varepsilon_2d)$-expander with $\delta(G) \ge d$. Let $U \subseteq V (G)$ satisfy
$|U| \le \exp((\log \log n)^2)$. Let $r = n^{1/8}$ and $\ell_ 0 = (\log \log n)^{20}$. Suppose $(A_i, B_i, C_i)$, $i \in [r]$, are such that the following hold for each $i \in [r]$.
\stepcounter{propcounter}
\begin{enumerate}[label = ({\bfseries \Alph{propcounter}\arabic{enumi}})]
\rm\item\label{App-ball-build1} $|A_i| \ge d_0$.

\rm\item\label{App-ball-build2} $B_i \cup C_i$ and $A_i$ are disjoint sets in $V (G) \backslash U$, with $|B_i|\le\frac{|A_i|}{ \log^{10} |A_i|}$.

\rm\item\label{App-ball-build3} $A_i$ has 4-limited contact with $C_i$ in $G - U- B_i$.

\rm\item\label{App-ball-build4} Each vertex in $B^{\ell_0}_{G-U-B_i-C_i}(A_i)$ has at most $d/2$ neighbors in $U$.

\rm\item\label{App-ball-build5} For each $j \in [r] \backslash \{i\}$, $A_i$ and $A_j$ are at least a distance $2\ell_ 0$ apart in $G- U- B_i- C_i-B_j- C_j$.
\end{enumerate}
Then $B^{\ell_0}_{G-U-B_i-C_i}(A_i)\ge \log^kn$ holds for some $i \in [r]$.
\end{lemma}

\begin{lemma}[Lemma 3.12 in \cite{2023liu}]\label{liu3.12}
For any $0 < \varepsilon_1, \varepsilon_2 < 1$, there exists $d_0=d_0(\varepsilon_1, \varepsilon_2)$ such that the
following holds for each $n \ge d \ge d_0$. Suppose that $G$ is an $n$-vertex bipartite
$(\varepsilon_1, \varepsilon_2d)$-expander with $\delta(G) \ge d$ and let $m = \frac{50}{\varepsilon_1 } \log^3 n$.
For any set $W \subseteq V (G)$ with $|W| \le \frac{\varepsilon_1n}{100 \log^2n}$, there is a set $B \subseteq G - W$ with size at least $\frac{n}{25}$ and diameter at most $2m$, and such that $G[B]$ is a $(D,m)$-expansion around some vertex $v \in B$ for $D = |B|$.
\end{lemma}
\begin{lemma}[Lemma 4.2 in \cite{2023liu}]\label{liu4.2}
For any $0 < \varepsilon_1 < 1$, $0 < \varepsilon_2 < 1/5$ and $k \in \mathbb{N}$, there exists $d_0=d_0(\varepsilon_1, \varepsilon_2, k)$ such that the following is true for each $n \ge d \ge d_0$. Suppose that $G$ is an $n$-vertex bipartite $(\varepsilon_1, \varepsilon_2d)$-expander with $\delta(G) \ge d-1$. Let $C$ be a shortest cycle in $G$ and let $x_1$, $x_2$ be distinct vertices in $V (G) \backslash V (C)$. Let $m = \frac{200}{\varepsilon_1}\log^3 n$ and $D \le \log^{5k}n$. Then, $G$ contains a $(D, m, 1)$-adjuster $(v_1, F_1, v_2, F_2, A)$ with $v_1 = x_1$, $v_2 = x_2$ and $V (C) \subseteq A$.
\end{lemma}
\begin{proof}[Proof of Lemma \ref{1-adjuster}]
Let $0 < \varepsilon_1 < 1$ be small enough that the property in Corollary \ref{coro} holds. Suppose that $G - U$ contains no $(D, m, 1)$-adjuster. Let $\Delta= 200mD$, $L =\{v \in V (G) : d(v) \ge\Delta\}$ and $G'= G -L$, so that $\Delta(G') \le \Delta$. Set  $\ell_0 = (\log \log n)^{20}$. Let $U_0 = \{v \in V (G) \backslash U : d(v, U) \ge d/2\}$. Let $A=U_0$ and $B=U$.  Note that $G[A, B]$ does not contain a copy of $K_{s,t}$ with $t$ vertices in $A$ and $s$ vertices in $B$. Then, by Corollary \ref{coro2.6} with $\delta_{\ref{coro2.6}}=d/2$, we have $|A|\le\left(\frac{|B|\cdot et}{d/2}\right)^s\le\log^{10sk}n$. Therefore, we can assume that $|U_0| \le D^{10s}$. As $\delta(G) \ge d$ and $n \ge d_0(\varepsilon_1, \varepsilon_2, k)$ is large, $G - U$ contains at least $(n-|U|-|U_0|) \cdot d/2/4 \ge nd/8$ edges. Let $U_1 = U \cup U_0$, so that $|U_1| \le 10D+D^{10s} \le20 \log^{10sk}n$.

Take a maximal collection $\mathbf{A_0}$ of adjusters in $G - U$, such that the following
hold.

\stepcounter{propcounter}
\begin{enumerate}[label = ({\bfseries \Alph{propcounter}\arabic{enumi}})]
\rm\item\label{App-adjuster-build1}
 The sets $V (F_1 \cup F_2)$, $(v_1, F_1, v_2, F_2, A) \in \mathbf{A_0}$, are subsets of $V (G')$ and are
all at least a distance $10\ell_0$ apart from each other and from $U_1 \backslash L$ in $G'$.

\rm\item\label{App-adjuster-build2} For each $\mathcal{A}\in \mathbf{A}_0$, for some $m_\mathcal{A}$ with $\log^3 d_0 \le m_\mathcal{A} \le m$, $\mathcal{A}$ is an $(m^2_\mathcal{A}, m_\mathcal{A}, 1)$-adjuster.
\end{enumerate}

The following three claims can be found in \cite{2023liu}. For the sake of completeness, we include their  proofs here.

\begin{claim}\label{claim1liu}
$|\mathbf{A}_0| \ge n^{1/4}$.
\end{claim}
\begin{proof}
Suppose that $|\mathbf{A}_0| < n^{1/4}$. Let $W = (U_1 \cup
(\cup _{\mathcal{A}\in \mathbf{A}_0} V (\mathcal{A}))\backslash L$. For each $\mathcal{A} = (v_1, F_1, v_2, F_2, A)$ $\in\mathbf{A}_0$, $|V (\mathcal{A})| = |F_1|+|F_2|+|A|\le2m^2_\mathcal{A} + 10m_\mathcal{A} \le 3m^2$, and therefore $|W| \le  n^{1/4} \cdot 3m^3 + 200 \log^{2k} n \le n^{1/3}$. Let $W'= B^{10\ell_0}_{G'}(W)$, so, as $\Delta(G') \le\Delta$, we have that $|W'| \le 2|W| \cdot\Delta^{10\ell_0} \le n^{1/2}$.

Now, there are at most $|W'|\Delta \le n^{1/2}\Delta \le \frac{nd}{16}$ edges in $G$ with some vertex in $W'$. Let $\bar{d}
= \frac{d}{64}$. As $G - U$ contains at least $nd/8$ edges, $G- U - W'$	 contains at least $\frac{nd}{16}$ edges, so that $d(G-U -W') \ge d/8=8\bar{d}$. Then, by Corollary \ref{coro}, $G- U - W’$ contains an $(\varepsilon_1, \varepsilon_2\bar{d})$-expander $H$ with $\delta(H) \ge\bar{d}$. Let $C$ be a shortest cycle in $H$. We will consider two cases, depending on how many vertices of $L$ there are in $V(H)\backslash V(C)$.

The first case is that  $|(V(H) \backslash V (C)) \cap L| \le 1$. Let $H': = H-(V(H)\backslash V(C))\cap L$. Then   $\delta(H') \ge \bar{d}-1$. Note that, for each $X\subseteq V (H)$ with
$\varepsilon_2\bar{d}/2 \le |X|\le |H'|/2 < |H|/2$,
we have
\begin{align*}
|N_{H'}(X)|&\ge|N_H(X)|-1 \ge |X| \cdot\varepsilon(|X|, \varepsilon_1, \varepsilon_2\bar{d})-1\\
&\ge|X| \cdot\varepsilon(|X|, \varepsilon_1,\varepsilon_2\bar{d})/2+\varepsilon_2\bar{d}/4\cdot\varepsilon(\varepsilon_2\bar{d}/2,\varepsilon_1,\varepsilon_2\bar{d})-1\\
&\ge |X| \cdot\varepsilon(|X|, \varepsilon_1/2, \varepsilon_2\bar{d})+\varepsilon_2\bar{d}/4\cdot\varepsilon_1/\log^2(15/2)-1\ge|X|\cdot\varepsilon(|X|,\varepsilon_1/2,\varepsilon_2\bar{d})
\end{align*}
where the last inequality follows as $\bar{d}\ge \frac{d_0(\varepsilon_1, \varepsilon_2, k)}{64}$ is large. So, $H'$ is an $(\varepsilon_1/2, \varepsilon_2\bar{d})$-expander with  $\delta(H') \ge \bar{d}-1$. Note that $C$ is a shortest cycle in $H'$.

Let $m_{H'}= 200 \log^3 |H'|/\varepsilon_1 \le m$, and note that, as $|H'| \ge \delta(H')+1 \ge \bar{d}\ge \frac{d_0}{64}$, and $d_0 = d_0(\varepsilon_1, \varepsilon_2, k)$ is large, $m_{H'} \ge \log^3 d_0$. Picking arbitrary vertices $x_1$, $x_2 \in V (H') \backslash V (C)$ and noting that $\bar{d}\ge \frac{d_0(\varepsilon_1, \varepsilon_2, k)}{64}$ is large, by Lemma \ref{liu4.2} with $(k, D)_{\ref{liu4.2}}= (10, m^2_{H'} )$, $H'$ contains an $(m^2_{H'} , m_{H'} , 1)$-adjuster $(v_1, F_1, v_2, F_2, A)$ with $V (C) \subseteq A$. As $A$ is disjoint from $V (F1 \cup F2)$, $V (C) \subseteq A $ and $(V (H') \backslash V (C)) \cap L = \emptyset$, we have that$ V (F_1 \cup F_2)$ is disjoint from $L$, and hence lies in $V (G')$. Together with $V (F_1\cup F_2) \subseteq V (H')$ being disjoint from $W'$ and so $10\ell_0$-far in $G'$ from the ends of the adjusters in $\mathbf{A}_0$ and from $U_1 \backslash L$, this violates the maximality of $\mathbf{A}_0$, a contradiction.

The other case is that $|(V (H) \backslash V (C)) \cap L| \ge 2$. We claim that there is  a $(D, 2m, 1)$-adjuster in $G-U$. Let $x_1$, $x_2 \in (V (H) \backslash V (C)) \cap L$ be distinct
and let $m_{H'} = 200 \log^3 |H'|/\varepsilon_1 \le m$. By Lemma  \ref{liu4.2} with $(k, D)_{\ref{liu4.2}} = (1, 1)$, $H$
contains a $(1, m_{H'} , 1)$-adjuster $(v_1, F_1, v_2, F_2, A)$ with $v_1 = x_1$ and $v_2 = x_2$. Using that $|A| \le 10m_{H'} \le 10m$, $|U|\le 10D$, and $d_G(x_1)$, $d_G(x_2) \ge \Delta = 200mD$, pick disjointly sets $X_1 \subseteq N_G(x_1) \backslash (U \cup A \cup \{x_2\})$ and $X_2 \subseteq N_G(x_2) \backslash (U \cup A \cup \{x_1\})$ with $|X_1| = |X_2| = D-1$. Letting $F_i'= G[\{x_i\} \cup X_i]$ for each $i \in [2]$, and noting $|A| \le 20m$, we have that $(x_1, F_1', x_2, F_2', A)$ is a $(D, 2m, 1)$-adjuster in $G-U$, a contradiction.

This completes the proof of Claim\ref{claim1liu}. 
\end{proof}
Now, let $\mathbf{A}_1 \subseteq \mathbf{A}_0$ be the set of adjusters $(v_1, F_1, v_2, F_2, A) \in \mathbf{A}_0$ for which there is no path with length at most $\ell_0$ from $V (F_1) \cup V (F_2)$ to $L \backslash U$ in $G - U - A$.
\begin{claim}\label{claim2liu}
$|\mathbf{A_1}| \ge n^{1/4}/2$.
\end{claim}
\begin{proof}[Proof of claim \ref{claim2liu}]
Let $r = n^{1/8}$. Suppose, for contradiction, that we can label distinct $\mathcal{A}_1, \ldots, \mathcal{A}_r \in \mathbf{A_0} \backslash \mathbf{A_1}$. For each $i \in [r]$, that $\mathcal{A}_i = (v_{i,1}, F_{i,1}, v_{i,2}, F_{i,2}, \bar{A_i})$ and let $P_i'$ be a shortest path with length at most  $\ell_0$ from $V (F_{i,1}) \cup V (F_{i,2})$ to $L \backslash U$ in $G-U - \bar{A_i}$. Relabelling, if necessary, for each $i \in [r]$ suppose the endvertex of $P_i'$ in $V (F_{i,1} \cup F_{i,2})$ is in $V (F_{i,1})$, and let $Q_i$ be a path from this endvertex of $P_i'$ to $v_{i,1}$ in $F_{i,1}$ with length at most $m_{\mathcal{A}_i}$.

For each $i \in [r]$, let $x_i$ be the endpoint of $P_i'$ in $L\backslash U$, and let $P_i = P_i'-x_i$. We shall apply Lemma \ref{liu3.7} by setting, for each $i \in [r]$, $A_i = V (F_{i,2})$, $B_i = \bar{A_i} \cup V (Q_i)\cup \{x_i\}$ and $C_i = V (P_i)$. By \ref{App-adjuster-build2}, $|A_i| = m^2_{\mathcal{A}_i} \ge \log^6 d_0$. Since  $d_0 = d_0(\varepsilon_1, \varepsilon_2, k)$ is large, we have that $|A_i| \ge d_{0\ref{liu3.7}}$, where $d_{0}$ is the function in Lemma \ref{liu3.7}, and thus  \ref{App-ball-build1} holds.

By \ref{Ad1}, we have that $V (F_{i,2})\cap V(F_{i,1})\cap(\bar{A_i})=\emptyset$. By \ref{App-adjuster-build1}, $V (F_{i,2}) \subseteq V (G') = V (G) \backslash L$. So that  $B_i \cup C_i$ and $A_i$ are disjoint. Since $B_i = \bar{A}_i \cup V (Q_i)\cup \{x_i\}$, $|B_i| \le |\bar{A}_i| + |Q_i| + 1 \le 20m_{\mathcal{A}_i} \le \frac{m^2_{\mathcal{A}_i}}{\log^{10}(m^2_{\mathcal{A}_i})}$ where the last inequality holds as $m_{\mathcal{A}_i} \ge \log^3(d_0(\varepsilon_1, \varepsilon_2, k))$ is large, and therefore \ref{App-ball-build2} holds.

As $P_i'$ is a shortest path from $V(F_{i,1}) \cup V (F_{i,2})$ to $L \backslash U$ in $G-U-\bar{A_i}$, which has an endvertex in $V(F_{i,1})$, and $A_i = V(F_{i,2})$, we have, for each  $\ell \in \mathbb{N}$, that $B^{\ell}_{G-U-\bar{A_i}}(A_i)$ has at most  $\ell+1$ vertices in $P_i'$, and hence $P_i$. Then, $A_i$ has 4-limited contact with $C_i$ in $G-U-\bar{A_i}$, and hence in $G-U-B_i$. So that \ref{App-ball-build3} holds.

Let $R_i$ be a path with length at most $10\ell_0$ from $A_i$ to $L\backslash(U\cup\{x_i\})$ in $G-U-B_i-C_i$. Then, there is a path $R_i' \subseteq R_i \cup F_{i,2}$ from $v_{i,2}$ to some vertex $y_i \in L\backslash (U \cup \{x_i\})$ with length at most $10\ell_0 + m_{\mathcal{A}_i} \le 2m-1$, and the path $Q_i\cup P_i'$ is a path from $v_{i,1}$ to $x_i$ with length at most $m_{\mathcal{A}_i} + \ell_0 \le 2m-1$ in $G-U-\bar{A_i}$ with vertices in $B_i \cup C_i$. Then, as $|U \cup A_i \cup V (R_i')\cup V (Q_i \cup P_i')| \le 10D + 10m_{\mathcal{A}_i} + 4m \le 10D+15m$, as $x_i$, $y_i \in L$ both have degree at least $\Delta = 200mD$, we can comfortably choose $X_i \subseteq N_G(x_i)$ and $Y_i \subseteq N_G(y_i)$ which are disjoint from each other and from $U \cup A_i \cup V (R_i') \cup V (Q_i \cup P_i')$ and have size $D -|P_i'\cup Q_i|$ and $D -|R_i'|$ respectively. Then, $(v_{i,1}, G[X_i \cup V (P_i') \cup V (Q_i)]$, $v_{i,2}$, $G[Y_i \cup V (R_i')], A_i)$ is a $(D, 2m, 1)$-adjuster in $G-U$, a contradiction. Therefore, there is no such path $R_i$. Consequently, recalling that $A_i = V (F_{i,2})$, we have
\[
B^{\ell_0}_{G-U-B_i-C_i}(A_i)=B^{\ell_0}_{G'-U-B_i-C_i}(A_i)
\]
which, by \ref{App-adjuster-build2}, is disjoint from $U_1$. By the choice of $U_0\subseteq U_1$, we have that \ref{App-ball-build4} holds.

Now, similarly, for any $j \in [r] \backslash \{i\}$, we have that
\[
B^{\ell_0}_{G-U-B_j-C_j}(A_j)=B^{\ell_0}_{G'-U-B_j-C_j}(A_j),
\]
so that, by \ref{App-adjuster-build1}, $B^{\ell_0}_{G-U-B_i-C_i}(A_i)$ and $B^{\ell_0}_{G'-U-B_j-C_j}(A_j)$ are disjoint. In particular, $A_i$ and $A_j$ are a distance at least $2\ell_0$ apart in $G-U-B_i-C_i-B_j-C_j$, and therefore \ref{App-ball-build5} holds.

Thus, by Lemma \ref{liu3.7}, there is some $j \in [r]$ for which $|B^{\ell_0}_{G-U-B_i-C_i}(A_j)|\ge\log^k n \ge D$. As $F_{j,2}$ is an $(m^{2}_{\mathcal{A}_j}, m_{\mathcal{A}_j})$-expansion of $v_{j,2}$ in $G-U-B_j-C_j$, $m_{\mathcal{A}_j}\le m$ and $A_j = V (F_{j,2})$, we have that $|B^{2m}_{G-U-B_i-C_i}(v_{j,2})|\ge D$ as  $\ell_0\ll m$. Therefore, by Proposition \ref{3.4}, we can pick a $(D, 2m)$-expansion, $F_{j,2}'$ say, of $v_{j,2}$ in $G-U-B_i-C_j$.

As $x_j \in L$, we can then pick a set $U'$ of neighbors of $x_j$ disjoint from $U\cup V (F'_{j,2})\cup\bar{A}_j\cup V(Q_j)\cup V(P_j')$ with $|U'|=D-|V(P_j'\cup Q_j)|$. Let $F_{j,1}' = G[U'\cup V (P_j')\cup V(Q_j)]$.
Note that $F'_{j,1}$ is then a $(D, 2m)$-expansion of $v_{j,1}$ as $Q_j \cup P_j'$ is a $(v_{j,1}, x_j)$-path with
length at most $m_{\mathcal{A}_j} +  \ell_0 \le 2m-1$. Finally, note that $(v_{j,1}, F_{j,1}', v_{j,2}, F_{j,2}', \bar{A}_j)$ is a $(D, 2m, 1)$-adjuster in $G-U$, a contradiction. Therefore, $|\mathbf{A}_0 \backslash \mathbf{A}_1| < r = n^{1/8}$, and so by Claim \ref{claim1liu}, we have $|\mathbf{A}_1| > n^{1/4}-r \ge n^{1/4}/2$.
\end{proof}

Let $\mathbf{A}_1' \subseteq \mathbf{A}_1$ satisfy $|\mathbf{A}'_1| = n^{1/4}/2$. Then, $|\cup_{\mathcal{A}\in \mathbf{A}_1\mathbf{'}} V (\mathcal{A})| \le n^{1/4} \cdot 3m^2 \le n^{1/3}$ by \ref{App-adjuster-build1}. Therefore,
\[
|U \cup B^{\ell_0}_{G'} (\cup_{\mathcal{A}\in A_1'} (V (\mathcal{A})\backslash L))| \le 10D + n^{1/3} \cdot 2\Delta^{\ell_0} \le n^{1/2}.
\]
Thus, by Lemma \ref{liu3.12}, there is a set $Z \subseteq V (G) \backslash U$ which has diameter at most $m/2$ and size $10m^2D$, and is a distance at least $\ell_0$ in $G'$	 from $V (\mathcal{A}) \backslash L$ for each $\mathcal{A} \in \mathbf{A}'_1$.

Let $\mathbf{A}_2\subseteq \mathbf{A}_1'$ be the set of adjusters $(v_1, F_1, v_2, F_2, A) \in \mathbf{A}_1'$ for which there is no path with length at most $m/2$ from $V (F_1) \cup V (F_2)$ to $Z$ in $G-U-A$.
\begin{claim}\label{claim3liu}
$|\mathbf{A_2}| \ge n^{1/4}/4$.
\end{claim}
\begin{proof}[Proof of claim \ref{claim3liu}.]
Let $r = n^{1/8}$. Suppose not,  and we can label distinct $\mathcal{A}_1,
\ldots, \mathcal{A}_r \in \mathbf{A_1'} \backslash \mathbf{A_2}$. Say, for each $i \in [r]$, that $\mathcal{A}_i = (v_{i,1}, F_{i,1}, v_{i,2}, F_{i,2}, \bar{A}_i)$ and let $P_i$ be a shortest path with length at most $m/2$ from $V (F_{i,1}) \cup V (F_{i,2})$ to $Z$ in
$G- U -\bar{A}_i$. Relabelling, if necessary, for each $ i\in [r]$ suppose the endvertex of $P_i$
in $V (F_{i,1} \cup F_{i,2}$) is in $V (F_{i,1})$, and let $Q_i$ be a path from this endvertex of $V (P_i)$
to $v_{i,1}$ in $F_{i,1}$ with length at most $m_{\mathcal{A}_i}$.

We will apply Lemma \ref{liu3.7} to $A_i = V (F_{i,2})$, $B_i = \bar{A}_i \cup V (Q_i)$ and $C_i = V (P_i)$, for each $i\in [r]$. For each $i \in [r]$, similarly to the proof of Claim \ref{claim2liu}, we have that \ref{App-ball-build1}-\ref{App-ball-build3} hold. By the choice of $\mathbf{A_1}$, for each $i \in[r]$, there is no path of length at most $\ell_0$ from $A_i$ to $L \backslash U$ in $G - U - B_i - C_i$. Therefore, the sets $B^{\ell_0}_{G-U-B_i-C_i}(A_i)$ and $B^{\ell_0}_{G'-U-B_i-C_i}(A_i)$ are the same set, and thus, by \ref{App-adjuster-build1}, this set is disjoint from $U_1$. Thus, \ref{App-ball-build4} holds by the definition of $U_1$. It similarly follows that $B^{\ell_0}_{G-U-B_i-C_i}(A_i)$ and $B^{\ell_0}_{G-U-B_j-C_j}(A_j)$ are vertex disjoint for each $j \in [r] \backslash \{i\}$, and thus \ref{App-ball-build5} holds.

Thus, by Lemma \ref{liu3.7}, there is some $j \in [r]$ for which $|B^{\ell_0}_{G'-U-B_i-C_i}(A_j)|=|B^{\ell_0}_{G-U-B_j-C_j}(A_j)|\ge D$. Thus, as $F_{j,2}$ is an $(m^2_{\mathcal{A}_j}, m_{\mathcal{A}_j})$-expansion of $v_{j,2}$ in $G'-U-B_j-C_j$ by \ref{App-adjuster-build1} and \ref{App-adjuster-build2}, and $A_j = V (F_{j,2})$, by Proposition \ref{3.4}, there is a $(D, 2m)$-expansion, $F'_{j,2}$ say, of $v_{j,2}$ in $B^{\ell_0}_{G'-U-B_j-C_j}(V(F_{j,2}))$. As $Z$ was chosen to have a distance at least $\ell_0$ in $G'$	 from $V (\mathcal{A}_j)\backslash L$, we have that $V(F'_{j,2})$ is disjoint from $Z$.

Now, as $Z$ has diameter at most $m/2$ in $G$, $Q_j \cup P_j \cup G[Z]$ is an expansion of $v_{j,1}$ with radius at most $\ell(Q_j ) +  \ell(P_j ) + m/2 \le 2m$ and size at least $D$. Therefore, by Proposition \ref{3.4}, we can find within $Q_j \cup P_j \cup G[Z]$ a $(D, 2m)$-expansion, $F'_{j,1}$ say, of $v_{j,1}$, which then must be vertex-disjoint from $\bar{A}_j$ and from $V(F'_{j,2})\subseteq B^{\ell_0}_{G'-U-B_j-C_j}(V(F_{j,2}))$. Thus, we have that $(v_{j,1}, F'_{j,1}, v_{j,2}, F'_{j,2}, \bar{A}_j)$ is a $(D, 2m, 1)$-adjuster in $G-U$, a contradiction. Thus, $|\mathbf{A_2}|\ge|\mathbf{A_1}| -r \ge n^{1/4}/4$, by
Claim \ref{claim2liu}.
\end{proof}
Let $r = n^{1/8}$. Using Claim \ref{claim3liu}, label distinct $\mathcal{A}_1, \ldots, \mathcal{A}_r \in \mathbf{A}_2$, and say, for each $i \in [r]$, that $\mathcal{A}_i = (v_{i,1}, F_{i,1}, v_{i,2}, F_{i,2}, \bar{A}_i)$. We shall apply Lemma \ref{liu3.7} to $A_i = V (F_{i,1} \cup F_{i,2})$, $B_i =\bar{A}_i$ and $C_i = \emptyset$. Similarly as in the proof of Claim \ref{claim3liu}, the only difference being that \ref{Ad3} holds trivially as $C_i = \emptyset$ and $A_i$ is slightly larger, we have that \ref{Ad3}-\ref{Ad4} hold.

Thus, by Lemma \ref{liu3.7}, there is some $j \in [r]$ with $|B^{\ell_0}_{G-U-B_j-C_j}(A_j)|=|B^{\ell_0}_{G-U-B_j}(A_j)|\ge10m^2D\ge10\log^3n|U\cup B_j|$.
Therefore, by Lemma \ref{diamter}, there is a path in $G - U - B_j$ from $B^{\ell_0}_{G-U-B_j}(A_j)$ to $Z$
with length at most $m/4$. Then, as $A_j = V (F_{j,1} \cup F_{j,2})$ and $B_j = \bar{A}_j$, there is a
path in $G-U-\bar{A}_j$ from $V (F_{j,1}\cup F_{j,2})$ to $Z$ with length at most $m/2$, contradicting
$\mathcal{A}_j \in \mathbf{A}_2$, and completing the proof.
\end{proof}
\end{document}